\documentclass{amsart}
\usepackage{amsfonts,amssymb,amsmath,amsthm}
\usepackage{url}
\usepackage{enumerate}
\newtheorem{thm}{Theorem}[section]
\newtheorem{cor}[thm]{Corollary}
\newtheorem{lem}[thm]{Lemma}
\newtheorem{prop}[thm]{Proposition}
\theoremstyle{definition}

\theoremstyle{remark}
\newtheorem{rem}[thm]{\bf Remark}
\newtheorem{exe}[thm]{\bf Example}

\numberwithin{equation}{section}

\begin{document}
\title{DYNAMICS OF NON ABELIAN AFFINE HOMOTHETIES GROUP  OF $\mathbb{C}^{n}$}

\author{Yahya N'dao and Adlene Ayadi}

 \address{Yahya N'dao, University of Moncton, Faculty of sciences, Moncton, Canada}
\address{Adlene Ayadi, University of Gafsa, Faculty of sciences, Department of Mathematics,Gafsa, Tunisia.}

\email{yahiandao@yahoo.fr ,\;\;\; adleneso@yahoo.fr}

\thanks{This work is supported by the research unit: syst\`emes dynamiques et combinatoire:
99UR15-15} \subjclass[2000]{37C85} \keywords{Homothety, Rotation,
orbit, density, minimal, non abelian, action, dynamic,...}

\begin{abstract}

In this paper, we study the action of non abelian group $G$
generated by affine homotheties  on\ $\mathbb{C}^{n}$. We prove
that there exist a  subgroup $\Lambda_{G}$
of $\mathbb{C}^{*}$, a $G$-invariant affine subspace $E_{G}$ of $\mathbb{C}^{n}$ and $a\in E_{G}$ such that
 $\overline{G(z)}=\overline{\Lambda_{G}}(z-a)+E_{G}$ for every $z\in \mathbb{C}^{n}$. In particular,
 $\overline{G(z)}=E_{G}$ for every $z\in E_{G}$ and if $E_{G}\neq\mathbb{C}^{n}$, every orbit in
$U=\mathbb{C}^{n}\backslash E_{G}$  is minimal in $U$. Moreover, we characterize the
 existence of dense orbit of $G$. As a consequence of the case $n=1$, we describe the action of affine rotations groups of $\mathbb{R}^{2}$.
\end{abstract}
\maketitle

\section{\bf Introduction }

 A map  $f :  \mathbb{C}^{n}\longrightarrow
\mathbb{C}^{n}$  is called an affine homothety  if there exists
$\lambda\in\mathbb{C}\backslash \{0,1\}$  and  $a\in
\mathbb{C}^{n}$  such that  $f(z)= \lambda (z-a)+a$  for every
 $z\in \mathbb{C}^{n}$. (i.e.  $f=T_{a}\circ (\lambda.
id_{\mathbb{C}^{n}})\circ T_{-a}$, \ $T_{a}:z\longmapsto z+a$, $id_{\mathbb{C}^{n}}$ the identity map of $\mathbb{C}^{n}$).  Write $f= (a,\lambda)$  and we call \ $a$ \
{\it the center} of  $f$ and \ $\lambda$ \ the ratio of $f$. \
\\
Denote by:\
\\
$\bullet$  $\mathcal{H}(n,\mathbb{K})$ the  group generated by all affine homotheties of
 $\mathbb{K}^{n}$ ($\mathbb{K}=\mathbb{R}$ or $\mathbb{C}$). i.e.
$$\mathcal{H}(n,\mathbb{K}):=\left\{\ f:z\
\longmapsto\ \lambda z+a;\ a\in\mathbb{K}^{n},\
\lambda\in\mathbb{K}^{*} \right\}.$$
\
\\
$\bullet$ $\mathcal{R}_{n}$ the  subgroup of $\mathcal{H}(n, \mathbb{C})$ generated by all affine rotations
  of $\mathbb{C}^{n}$. i.e. $$\mathcal{R}_{n}:=\left\{\ f:z\
\longmapsto\ e^{i\theta} z+a;\ a\in\mathbb{C}^{n},\
\theta\in\mathbb{R} \right\}.$$
\
\\
$\bullet$ $\mathcal{T}_{n}$  the subgroup of $\mathcal{H}(n, \mathbb{C})$ generated by
all translations of $\mathbb{C}^{n}$.
\
\\
$\bullet$ Let $H_{2}=\left(\frac{\pi}{2}+\pi\mathbb{Z}\right)\cup (\pi\mathbb{Z})$, $F_{2}=\{e^{ix},\ \ x\in H_{2}\}$ and
  $$\mathcal{S}_{2}\mathcal{R}_{n}:=\left\{\ f=(a,e^{i\theta})\in\mathcal{R}_{n};\ \theta\in H_{2},a\in\mathbb{C}^{n}\right\}.$$\
  \\
 $\bullet$ Let $H_{3}=\left(\frac{\pi}{3}+\pi\mathbb{Z}\right)\cup \left(-\frac{\pi}{3}+\pi\mathbb{Z}\right)\cup (\pi\mathbb{Z})$, $F_{3}=\{e^{ix},\ \ x\in H_{3}\}$ and
  $$\mathcal{S}_{3}\mathcal{R}_{n}:=\left\{\ f=(a,e^{i\theta})\in\mathcal{R}_{n};\ \theta\in H_{3},a\in\mathbb{C}^{n}\right\}.$$\
\\
We have $F_{2}$ and $F_{3}$ are finite, $\mathcal{S}_{2}\mathcal{R}_{n}$ and $\mathcal{S}_{3}\mathcal{R}_{n}$ are subgroups of $\mathcal{H}(n, \mathbb{C})$
containing $\mathcal{T}_{n}$.
\\
 $\bullet$ $\mathcal{SR}_{n}:=\mathcal{S}_{2}\mathcal{R}_{n}\cup\mathcal{S}_{3}\mathcal{R}_{n}$.\
\\
\\
 We say a {\it group of affine
homotheties} of $\mathbb{C}^{n}$ any subgroup of $\mathcal{H}(n,\mathbb{C})$.\
\bigskip

 Let $G$ be a non abelian  subgroup of
$\mathcal{H}(n,\mathbb{C})$. There is a natural action
$\mathcal{H}(n,\mathbb{C})\times \mathbb{C}^{n} :\
\longrightarrow\ \mathbb{C}^{n}$.\ $(f, v)\ \longmapsto\ f(v)$.
For a vector $v \in \mathbb{C}^{n}$, denote by $G(v) := \{f(v):\  f
\in G \}\subset \mathbb{C}^{n}$ the \emph{orbit} of $G$ through
$v$. A subset  $A \subset \mathbb{C}^{n}$  is called
\emph{$G$-invariant} if $f(A)\subset A$  for any $f\in G$; that
is  $A$  is a union of orbits and denote by $\overline{A}$
(resp. $\overset{\circ}{A}$ ) the closure (resp. interior) of $A$.
\\
If  $U$  is an
 open  $G$-invariant set, the orbit  $G(v)\subset U$ \ is
 called \emph{minimal in  $U$} if  $\overline{G(v)}\cap
U = \overline{G(w)}\cap U$  for every  $w\in \overline{G(v)}\cap
U$.
\medskip

 We say that $H$ is an {\it affine subspace} of
$\mathbb{C}^{n}$ with dimension $p$ if  $H=E+a$, for some $a\in\mathbb{C}^{n}$
and some vector subspace $E$ of $\mathbb{C}^{n}$
with dimension $p$. For every subset $A$ of $\mathbb{C}^{n}$,
denote by $vect(A)$ (resp. $Aff(A)$) the vector (resp. affine) subspace of $\mathbb{C}^{n}$
generated by all elements of $A$.
\
\\
\\
 Denote by:\
 \\
 - $\Lambda_{G}:=\{\lambda: \ f=(a,\lambda)\in
G\}$. It is obvious that $\Lambda_{G}$ is a subgroup of $\mathbb{C}^{*}$ (see Lemma ~\ref{L:00}).
 \\
- $\mathrm{Fix}(f):=\{z\in \mathbb{C}^{n}: \
f(z)=z\}$,  for every $f\in \mathcal{H}(n,
\mathbb{C})$. See that $Fix(f)=\emptyset$ if $f\in \mathcal{T}_{n}$ and $Fix(f)=a$ if $f=(a,\lambda)\in
 \mathcal{H}(n, \mathbb{C})\backslash \mathcal{T}_{n}$.
\\
 - $\Gamma_{G}:=
 \underset{f\in
G\backslash\mathcal{T}_{n}}{\bigcup}\mathrm{Fix}(f).$  Since $G$ is non abelian then
$G\backslash\mathcal{T}_{n}\neq\emptyset$, so $\Gamma_{G}\neq\emptyset$.
\\
- $G_{1}:=G\cap \mathcal{T}_{n}$ is a subgroup
of  $\mathcal{T}_{n}$.
\
\\
- $G_{1}(0)=\{f(0),\ f\in G_{1}\}$.
\
\\
- \ $E_{G}=Aff(\Gamma_{G}\cup G_{1}(0))$ the  affine subspace of
$\mathbb{C}^{n}$ generated by  $\Gamma_{G}\cup G_{1}(0)$.
\
\\
- \ $S^{1}=\{z\in\mathbb{C}:\ \  |z|=1\}$.
\\
Remark that $E_{G}\neq\emptyset$ since $\Gamma_{G}\neq\emptyset$ and $G_{1}(0)\neq\emptyset$.
\medskip

In \cite{AYN}, we have described the action of non abelian subgroup of $\mathcal{H}(n,\mathbb{R})$.
This paper can be viewed as continuation of that work.
\
\\
I learned that Zhukova  have proved in \cite{NZh} similar results to Lemma ~\ref{L:9},
Proposition ~\ref{p:1} and Corollary ~\ref{C:1}.(ii), in the real case.
 The methods of proof in \cite{NZh} and in this paper are quite different and have different consequences.
\medskip

 In \cite{AG}, Arek Goetz investigates noninvertible piecewise isometries in $\mathbb{R}^{2}$
  with the particular interest on the maximal invariant sets and $\omega$-limit sets.
   Unlike in \cite{MB}, the
induced isometries $T_{0}$ and $T_{1}$ of its system $T : X
\longrightarrow X$ are not translations but rotations. The
partition $P$ consists of two atoms: $P_{0}$, the open left
halfplane and $P_{1}$, the closed right halfplane.

\medskip

 Our principal results can be stated as
follows:

\begin{thm}\label{T:1} Let  $G$ be a non abelian subgroup of
$\mathcal{H}(n,\mathbb{C})$ such that $\Lambda_{G}\backslash\mathbb{R}\neq\emptyset$. Then :\\ $(1)$ If
$G\backslash\mathcal{SR}_{n}\neq\emptyset$, one has:
\begin{itemize}
  \item [(i)]  $\overline{G(z)}=E_{G}$, for every $z\in E_{G}$.
   \item [(ii)] if $U=\mathbb{C}^{n}\backslash E_{G}\neq\emptyset$, there exists $a\in E_{G}$ such that
    $\overline{G(z)}=\overline{\Lambda_{G}}(z-a)+E_{G}$, for every $z\in U$.
   \end{itemize}
$(2)$ If  $G\subset\mathcal{SR}_{n}$,  one has:\
\begin{itemize}
\item[(i)] $G\subset\mathcal{S}_{i}\mathcal{R}_{n}$ for some $i\in\{2,3\}$.\
\item[(ii)] $\overline{G(z)}=F_{i}z+\overline{G(0)}$, for every $z\in \mathbb{C}^{n}$.\
\end{itemize}
\end{thm}
\medskip

\begin{cor}\label{C:1}Under notations of Theorem ~\ref{T:1}. One has:\\
(1)  If
$G\backslash\mathcal{SR}_{n}\neq\emptyset$ and $U\neq\emptyset$, then:
\begin{itemize}
   \item [(i)] Every orbit in $U$ is minimal in
  $U$.
   \item [(ii)]   If $G\backslash \mathcal{R}_{n}\neq\emptyset$, then $E_{G}$  is a minimal set of \ $G$  in \
$\mathbb{C}^{n}$ \ contained in the closure of every orbit of $G$.
  \item [(iii)] All orbits in $U$ are homeomorphic.
\end{itemize}
$(2)$ If
$G\subset\mathcal{SR}_{n}$, then  $G$ has a dense orbit if and only if $\overline{G_{1}(0)}=\mathbb{C}^{n}$.
\end{cor}
\medskip

 \begin{cor} \label{C:2} Let $G$  be a non abelian subgroup of  $\mathcal{H}(n,
 \mathbb{C})$ with $\Lambda_{G}\backslash\mathbb{R}\neq\emptyset$ and $G\backslash \mathcal{SR}_{n}\neq \emptyset$,
  then $G$ has no discrete orbit.
  \end{cor}
\medskip

 \begin{cor}\label{C:3} Let $G$  be a non abelian subgroup of  $\mathcal{H}(n, \mathbb{C})$
 with $\Lambda_{G}\backslash\mathbb{R}\neq\emptyset$ and $G\backslash \mathcal{SR}_{n}\neq \emptyset$. Then the following assertions are equivalents:
\begin{itemize}
\item [(1)] $G$ has a dense orbit in  $\mathbb{C}^{n}$.
  \item [(2)] Every orbit
of  $U$  is dense in
$\mathbb{C}^{n}$.
  \item [(3)] $G$ satisfies one of the following:
\begin{itemize}
  \item [(i)] $E_{G}=\mathbb{C}^{n}$
  \item [(ii)] $dim(E_{G})=n-1$ and  $\overline{\Lambda_{G}}=\mathbb{C}$.
    \end{itemize}
\end{itemize}
\end{cor}
\medskip

\begin{thm}\label{TT:1} Let  $G$ be a non abelian  group generated by two affine rotations $R_{\theta}$ and $R'_{\theta'}$
of $\mathbb{R}^{2}$, having angle respectively $\theta$ and $\theta'$.
Then:\\ $(1)$  every orbit of $G$ is dense in $\mathbb{R}^{2}$ if and only if there is one of the following:\
\begin{itemize}
    \item [(i)] $\left(H_{2}\cup H_{3}\right)\backslash\{\theta,\theta'\}\neq\emptyset$.
    \item [(ii)] $\theta\in H_{i}$ and  $\theta'\in H_{j}$ with $i\neq j$, $i,j\in\{2,3\}$.
  \end{itemize}
$(2)$ every orbit of $G$ is closed and discrete in $\mathbb{R}^{2}$ if and only if  $G_{1}(0)$ is closed
and discrete with $\theta,\theta'\in H_{i}$ for some $i\in\{2,3\}$.
\end{thm}
\medskip

\begin{rem} If $G$ is a non abelian subgroup of $\mathcal{H}(n, \mathbb{C})$ with
$\Lambda_{G}\subset\mathbb{R}$, then it can be considered as a subgroup of  $\mathcal{H}(2n, \mathbb{R})$, by identifying $\mathbb{C}^{n}$
 to $\mathbb{R}^{2n}$. So Theorem ~\ref{T:1} has the same form of Theorem1.1 in the real case (see ~\cite{AYN}).
\end{rem}
\medskip

 \begin{cor} \label{C:4} If $G$ is a non abelian subgroup of  $\mathcal{H}(n,
 \mathbb{C})$ generated by n-2 affine maps, it has no dense orbit.
  \end{cor}
\medskip

This paper is organized as follows: In Section 2, we introduce some preliminaries Lemmas. In Section 3,
we characterize the case $n=1$ and  we prove Theorem ~\ref{TT:1}.
 Section 4 is devoted to given some results in the case $G\backslash\mathcal{SR}_{n}\neq\emptyset$. In Section 5,
  we characterize any subgroup of $\mathcal{SR}_{n}$. In Section 6, we prove
Theorem ~\ref{T:1}, Corollaries ~\ref{C:1}, ~\ref{C:2}, ~\ref{C:3} and ~\ref{C:4}. In Section 7, we give three examples.
\medskip

\section{\textbf{Preliminaries Lemmas}}

\begin{lem}\label{L:6}\ \begin{itemize}
\item [(i)]Let $f=(a,\alpha),\ g=(b, \beta)\in\mathcal{H}(n, \mathbb{C})\backslash\mathcal{T}_{n}$ then
$f\circ g=g\circ f$ if and only if $a=b$ or $\alpha=1$ or
$\beta=1$.\
\item [(ii)] If $Fix(f)=Fix(g)$ then $f\circ g=g\circ f$.\
\item [(iii)]  Let $G$ be a non abelian subgroup of $\mathcal{H}(n, \mathbb{C})$, then there exist
$f=(a,\alpha)$, $g=(b, \beta)\in G\backslash\mathcal{T}_{n}$ such
that $a\neq b$.
\end{itemize}
\end{lem}
\medskip

\begin{proof} (i) $f\circ g(x)=g\circ
f(x)$, for every $x\in \mathbb{R}^{n}$, if and only if
\begin{align*}
 \lambda(\mu(x-b)+b-a)+a& =\mu(\lambda(x-a)+a-b)+b,\\
\Longleftrightarrow\ \ \ \ \ \ \ \  \ \ \  \ -\lambda\mu
b+\lambda(b-a)+a&  =-\mu\lambda a+\mu(a-b)+b,\\
\Longleftrightarrow\ \ \ \ \ \ \ \
(a-b)(\lambda\mu-\lambda-\mu+1)& =0\\
 \Longleftrightarrow \ \ \ \ \ \ \ \ \ \  \ (a-b)(\lambda-1)(\mu-1)&
 =0.
\end{align*} So the results follows.
\\
\\
(ii) There are two cases:\
\\
$\bullet$ If $Fix(f)=Fix(g)=\emptyset$ then $f=T_{a}$ and $g=T_{b}$ for some $a,b\in \mathbb{R}^{n}$, so $f\circ g=g\circ f$.
\\
$\bullet$ If $Fix(f)=Fix(g)=a$, then $T_{a}\circ f \circ T_{-a}=\lambda id$ and $T_{a}\circ g \circ T_{-a}=\mu id$,
  for some $\lambda,\mu\in \mathbb{C}^{*}$, so $f\circ g=g\circ f$.\
\\
The proof of (iii) results from (ii) since $G$ is non abelian.
\end{proof}
\medskip

\begin{lem} \label{L:4}$($\cite{AYN}, Lemma 2.3$)$ Let   $\mathcal{B}=(a_{1},\dots,a_{n})$  be a basis of  $\mathbb{C}^{n}$.
 Then  $\mathcal{A}ff(\mathcal{B})$  is defined by
 $$\mathcal{A}ff(\mathcal{B}):=\left\{z=\underset{k=1}{\overset{n}{\sum}}\alpha_{k}a_{k}:
  \ \alpha_{k}\in \mathbb{C},\
  \underset{k=1}{\overset{n}{\sum}}\alpha_{k}=1\right\}.$$
\end{lem}
\medskip

\begin{rem}\label{r:3} As consequence of Lemma ~\ref{L:4}, if $E_{G}$ contains $a,a_{1},\dots,a_{n}$  such that $(a_{1},\dots,a_{n})$
 is a basis of  $\mathbb{C}^{n}$ and
 $a=\underset{k=1}{\overset{n}{\sum}}\alpha_{k}a_{k}$ with
 $\underset{k=1}{\overset{n}{\sum}}\alpha_{k}\neq 1.$ Then
 $E_{G}=\mathbb{C}^{n}$.
\end{rem}
\medskip

\begin{lem}\label{L:00} Let $G$ be a non abelian subgroup of $\mathcal{H}(n, \mathbb{C})$.
 Then $\Lambda_{G}$ is a subgroup of $\mathbb{C}^{*}$. Moreover,  $0\in \overline{\Lambda_{G}}$ if $G\backslash \mathcal{R}_{n}\neq\emptyset$.
 \end{lem}
 \medskip

\begin{proof} One has $1\in \Lambda_{G}$ since $id_{\mathbb{R}^{n}}\in G$. Let $\lambda, \mu\in \Lambda_{G}$
and $f,g\in G$ defined by $f:x\longmapsto \lambda x+a$, and
$g:x\longmapsto \mu x+b$, $x\in \mathbb{R}^{n}$, so $f\circ
g^{-1}(x)=\frac{\lambda}{\mu}x-\frac{\lambda b}{\mu}+a$.
 Hence $\frac{\lambda }{\mu}\in \Lambda_{G}$.\ Moreover, $\Lambda_{G}\backslash S^{1}\neq\emptyset$, if $G\backslash \mathcal{R}_{n}\neq\emptyset$.
 So $\underset{m\to \pm\infty}{lim}\alpha^{m}=0$, for any $\alpha\in\Lambda_{G}\backslash S^{1}$. It follows that $0\in \overline{\Lambda_{G}}$.
\end{proof}
\
\\
\begin{lem}\label{L:1} Let  $G$ be a non abelian subgroup of
$\mathcal{H}(n,\mathbb{C})$ with  $G\backslash
\mathcal{SR}_{n}\neq\emptyset$ and
$\Lambda_{G}\backslash\mathbb{R}\neq\emptyset$.  Then:
\begin{itemize}
\item[(i)] $\Lambda_{G}\backslash(F_{2}\cup
F_{3}\cup\mathbb{R})\neq\emptyset$.
  \item [(ii)] if $E_{G}$ is a vector space, there exist
$f_{1}=(a_{1},\lambda),\dots,f_{p}=(a_{p},\lambda)\in G\backslash \mathcal{SR}_{n}$  such that
$\lambda\in\Lambda_{G}\backslash\mathbb{R}$ and $(a_{1},\dots,a_{p})$ is a basis of $E_{G}$.
  \item [(iii)] if $G'=T_{-a}\circ G \circ T_{a}$ for some  $a\in\Gamma_{G}$, then $E_{G'}=T_{-a}(E_{G})$ and $\Lambda_{G'}=\Lambda_{G}$.
 \end{itemize}
\end{lem}
\medskip

\begin{proof}(i)
Let $f=(a,\lambda'), g=(b,\mu)\in G$ such that $\lambda'\in\Lambda_{G}\backslash\mathbb{R}$ and $\mu\notin\left(F_{2}\cup F_{3}\right)$.  Suppose that
 $\lambda'\in F_{i}$ for some $i\in\{2,3\}$ and $\mu\in \mathbb{R}$, so $|\lambda'|=1$ and $|\mu|\neq 1$ since $-1,1\in F_{2}\cup F_{3}$. Then
 $|\lambda'\mu|\neq 1$, it follows that $\lambda'\mu \notin\left(F_{2}\cup F_{3}\right)$ and  $\lambda'\mu \notin\mathbb{R}$ because $\mu\in \mathbb{R}$ and
 $\lambda'\notin \mathbb{R}$. Therefore
$f\circ g=(c, \lambda'\mu)$ for  $c=\frac{-\lambda'\mu b+\lambda'(b-a)+a}{1-\lambda'\mu}$, so  $\lambda=\lambda'\mu\in\Lambda_{G}\backslash \left(F_{2}\cup F_{3}\cup \mathbb{R}\right)$.
\
\\
\\
(ii) Suppose that $\mathrm{dim}(vect(\Gamma_{G}))=k<
\mathrm{dim}(E_{G})=p$. By Lemma~\ref{L:6} (iii), $k\geq 1$. By
(i), we let $a_{1},\dots,a_{k}\in \Gamma_{G}$ and
$b_{k+1},\dots,b_{p}\in G_{1}(0)$,  such that:\
\\
- $\mathcal{B}_{1}=(a_{1},\dots,a_{k},b_{k+1},\dots,b_{p})$  is a basis
of  $E_{G}$.\
\\
-  $(a_{1},\dots,a_{k})$ is a basis of $vect(\Gamma_{G})$.\
\\
- $f=(a_{1},\lambda)\in G\backslash \mathcal{SR}_{n}$ with
$\lambda\in \Lambda_{G}\backslash\mathbb{R}$.\ (i.e. $\lambda\in
\Lambda_{G}\backslash(F_{2}\cup F_{3}\cup\mathbb{R})$).
\\
-$f_{k}=(a_{i},\lambda_{i})\in G\backslash
\mathcal{T}_{n}$, $i=2,\dots, k$.\
\\
\\
Write $f'_{i}=f_{i}\circ f\circ f^{-1}_{i}$, for every  $2\leq i
\leq k$. We have $f'_{i}=(f_{i}(a_{1}), \lambda)\in G\backslash
\mathcal{T}_{n}$. See that
$f_{i}(a_{1})=\lambda_{i}a_{1}+(1-\lambda_{i})a_{i}\in\Gamma_{G}$,
$i=2,\dots,k$.
\\
\\
 For every  $k+1\leq i \leq p$  there exists
$T_{i}\in G_{1}$  such that  $T_{i}(0)=b_{i}$. Write
$f_{i}=T_{i}\circ f\circ T^{-1}_{i}$, for every  $k+1\leq i \leq
p$. We have $f_{i}=(T_{i}(a_{1}), \lambda)\in G\backslash
\mathcal{SR}_{n}$.  So $T_{i}(a_{1})\in \Gamma_{G}$ and
$T_{i}(a_{1})=a_{1}+b_{i}$, \ $k+1\leq i \leq p$.
\
\\
 \\ Let's show that
$\mathcal{B}_{2}=(f_{1}(a_{1}),\dots,f_{k}(a_{1}),T_{k+1}(a_{1}),\dots,T_{p}(a_{1}))$
 is a basis of $E_{G}$: \\
 Let $$M=\left[\begin{array}{cc}
            B & A \\
            0 & I_{p-k}
          \end{array}\right]\in M_{n}(\mathbb{C}), \ \ \ \mathrm{with} \ \ \ A=\left[\begin{array}{ccc}
                  1 & \dots  & 1 \\
                  0 & \dots  & 0 \\
                   \vdots & \ddots  & \vdots  \\
                  0 & \dots   & 0
                \end{array}
\right]\in M_{k,p-k}(\mathbb{C}), $$
$$B=\left[\begin{array}{ccccc}
                  1& \lambda_{2} & \lambda_{3}& \dots  & \lambda_{k} \\
                  0 & 1-\lambda_{2}& 0 & \dots & 0 \\
                   \vdots & \ddots & \ddots &\ddots & \vdots  \\
                    \vdots & \ & \ddots &\ddots & 0  \\
                   0 & \dots & \dots   & 0& 1-\lambda_{k}
                \end{array}
\right]\in M_{k}(\mathbb{C}).$$
\\
 and  $I_{p-k}$ is the identity matrix of $M_{p-k}(\mathbb{C})$. As
 $f_{i}\in G\backslash \mathcal{T}_{n}$, $\lambda_{i}\neq 1$, for every $i=2,\dots,k$, so $M$ is invertible and $M(\mathcal{B}_{1})=\mathcal{B}_{2}$. So
$\mathcal{B}_{2}$ is a basis of $E_{G}$ contained in
$\Gamma_{G}$, a contradiction. We conclude that $k=p$.
\
\\
\\
(iii) Suppose that  $E_{G}$  is an affine subspace of
 $\mathbb{R}^{n}$  with dimension $p$. Let $a\in \Gamma_{G}$ and
 $G'=T_{-a}\circ G \circ T_{a}$.  Set $f=(a,\lambda)\in G\backslash \mathcal{T}_{n}$, then
 $T_{-a}\circ f\circ T_{a}=\lambda id_{\mathbb{C}^{n}}\in G'\backslash\mathcal{T}_{n}$, so
 $0\in\Gamma_{G'}\subset E_{G'}$, hence $E_{G'}$ is a vector space. By $(ii)$ there exists a basis
$(a'_{1},\dots,a'_{p})$  of  $E_{G'}$  contained in
$\Gamma_{G'}$.  Since  $\Gamma_{G'}=T_{-a}(\Gamma_{G})$, we let
 $a_{k}=T_{a}(a'_{k})$,  $1\leq k \leq p$, then $a_{1},\dots,a_{p}\in\Gamma_{G}$. We have  $\Gamma_{G'}=T_{-a}(\Gamma_{G})\subset T_{-a}(E_{G})$  and  $T_{-a}(E_{G})$  is a
  vector subspace of  $\mathbb{R}^{n}$  with dimension $p$, containing  $a'_{1},\dots,a'_{p}$. Therefore  $E_{G'}=T_{-a}(E_{G})$.
    \\
  On the other hand,  for every $g=(b,\mu)\in G\backslash \mathcal{T}_{n}$, $T_{-a}\circ g \circ T_{a}=(b-a,\mu)$, so $\Lambda_{G'}=\Lambda_{G}$.\
\end{proof}
\medskip

\begin{lem}\label{L:020} Let $G$ be a non abelian subgroup of $\mathcal{H}(n,\mathbb{C})$.
Then:
\begin{itemize}
  \item [(i)] If $E_{G}$ is a vector
subspace of $\mathbb{C}^{n}$, then $G(0)\subset E_{G}$.
  \item [(ii)] $\Gamma_{G}$ and $E_{G}$ are  $G$-invariant.
\end{itemize}
\end{lem}
\medskip

\begin{proof}(i) By construction, $G_{1}(0)\subset E_{G}$. Let $f\in G\backslash G_{1}$, then  $f=(a,\lambda)$, for
some $\lambda\in\Lambda_{G}$  and \ $a\in\Gamma_{G}\subset E_{G}$. Therefore \
$f(z)=\lambda(z-a)+a$, $z\in \mathbb{C}^{n}$  and
$f(0)=(1-\lambda)a$, so $f(0)\in E_{G}$ since $E_{G}$ is a vector space.
\
\\
\\
(ii) $\Gamma_{G}$ is $G$-invariant: Let  $a\in \Gamma_{G}$ and \
$g\in G$  then there exists $\lambda\in\mathbb{C}\backslash
(F_{2}\cup F_{3})$ such that $f=(a,\lambda)\in G\backslash
\mathcal{SR}_{n}$. We let \ $h=g\circ f\circ
g^{-1}=(g(a),\lambda)\in G\backslash \mathcal{SR}_{n}$, so
$g(a)\in \Gamma_{G}$ \ and hence $\Gamma_{G}$  is $G$-invariant.
\
\\
 $E_{G}$  is $G$-invariant: Let $a\in \Gamma_{G}$
and $G'=T_{-a}\circ G \circ T_{a}$. We have  $G'$ is a non abelian
subgroup of  $\mathcal{H}(n, \mathbb{C})$  and
$E_{G'}=T_{-a}(E_{G})$  is a vector subspace of $\mathbb{C}^{n}$.
Let  $f\in G'$ having the form $f(z)=\lambda z+b$, $z\in
\mathbb{C}^{n}$.  By (i), $b=f(0)\in\Gamma_{G'} \subset E_{G'}$.
So for every  $z\in E_{G'}$, $f(z)\in E_{G'}$, hence $E_{G'}$ is
$G'$-invariant. By Lemma ~\ref{L:1}.(iii) one has
$E_{G}=T_{-a}(E_{G'})$ is $G$-invariant.
\end{proof}
\bigskip

\begin{lem} \label{L:1200} Let  $G$\ be a non abelian subgroup of
$\mathcal{H}(n,\mathbb{C})$. Suppose that $E_{G}$ is a vector
space and there exist
$f_{1}=(a_{1},\lambda),\dots,f_{p}=(a_{p},\lambda)\in
G\backslash\mathcal{SR}_{n}$  with
$\lambda\in\Lambda_{G}\backslash\{0,1\}$ and
$\mathcal{B}_{1}=(a_{1},\dots,a_{p})$   is a bases of $E_{G}$.
Then there exists \ $f=(a,\lambda)\in G\backslash
\mathcal{SR}_{n}$  such that
$\mathcal{B}_{2}=(a_{1}-a,\dots,a_{p}-a)$  is also a basis of
$E_{G}$.
\end{lem}
\medskip

\begin{proof} Let
$a=f_{p-1}(a_{p})$, since  $f_{p-1}=(a_{p-1},\lambda)\in G\backslash \mathcal{SR}_{n}$ then $a=\lambda a_{p}+(1-\lambda)a_{p-1}.$
 By Lemma ~\ref{L:020}.(ii), $\Gamma_{G}$ is $G$-invariant, so $a\in \Gamma_{G}$.
Let
$$P=\left[\begin{array}{cccccc}
            1 & 0 & \dots & \dots & 0 & 0 \\
            0 & \ddots & \ddots & \ddots & \vdots & \vdots \\
            \vdots & \ddots & \ddots & \ddots & \vdots & \vdots \\
            0 & \dots & 0& 1& 0 & 0 \\
            \lambda-1 & \dots & \dots & \lambda-1 & \lambda & \lambda-1\\
            -\lambda & \dots & \dots & -\lambda &\lambda & 1-\lambda
          \end{array}
\right].$$
Since $\lambda\notin\{0,1\}$, then
$\mathrm{det}(P)=2\lambda(1-\lambda)\neq 0$ and $P$  is invertible. We have  $P(\mathcal{B}_{1})=\mathcal{B}_{2}$
so  $\mathcal{B}_{2}$  is a basis of  $\mathbb{R}^{n}$.
\end{proof}
\bigskip

By the same proofs of Lemmas 2.8 and  2.1, in ~\cite{AYN}, we can show the following Lemma:

\begin{lem}\label{L:1111} Let $G$ be the subgroup of $\mathcal{H}(n, \mathbb{C})$ generated by $f_{1}=(a_{1},\lambda_{1}),\dots,
 f_{p}=(a_{p},\lambda_{p})\in \mathcal{H}(n, \mathbb{C})\backslash \mathcal{SR}_{n}$. Then $E_{G}=\mathcal{A}ff(\{a_{1},\dots, a_{p}\})$.
\end{lem}
\bigskip

\begin{lem}\label{L:222} Let $G$ be a subgroup of $\mathcal{H}(n, \mathbb{C})$ generated by $f=(a, \lambda)$ and $g=(b,\mu)$.
Then $\Delta=\mathbb{C}(b-a)+a$ is $G$-invariant and $G_{/\Delta}$ is a subgroup of $\mathcal{H}(1, \mathbb{C})$.
\end{lem}
\medskip

\begin{proof} Let $\alpha\in \mathbb{C}$, and  $z=\alpha(b-a)+a$  we have

$$\begin{array}{ccc}
  f(z) =\lambda(\alpha(b-a)+a-a)+a & \ \ \ \ \ \mathrm{and}\     & g(z) =\mu(\alpha(b-a)+a-b)+b\\
   =\lambda\alpha(b-a)+a  \ \ \ \ \ \ \ & \ \ \ \ \   &\ \ \ \ \ \ \ \ \ \ \ \ \ =\mu(\alpha-1)(b-a)+b-a+a.\\
  \ \ \ \ \ \ \ \ \ \ \ & \ \ \ \   & \ \ \ \ \ \ \ \ \ \  =(1+\mu(\alpha-1))(b-a)+a
\end{array}$$

\medskip

  So  $f(z),\ g(z)\in \mathbb{C}(b-a)+a$.
\end{proof}
\medskip

\begin{lem}\label{L:bb10} Let $G$ be the group generated by $h=\lambda Id_{\mathbb{C}}$  and  $f=(a,\lambda)$
 with  $a\in \mathbb{C}^{*}$, $\lambda\notin \mathbb{C}\backslash\{0,1\}$.
 Then for every $k\in\mathbb{Z}^{*}$, one has:\
 \\
 $(i)$ $\Lambda_{G}=\{\lambda^{j},\ \
 j\in\mathbb{Z}\}$ and for every $b\in\Gamma_{G}$, $g=(b,\lambda)\in G$.\
 \\
 $(ii)$   $(\lambda^{k}-1)^{2} G_{1}(0)\subset
 G_{1}(0)$ and $ (\lambda^{k}-1)^{2} \Gamma_{G}\subset
G_{1}(0)$.\
 \\
 $(iii)$ $\lambda^{k}G_{1}(0)\subset
 G_{1}(0)$.\
 \\
 $(iv)$ if  $\lambda^{k}\neq1$,  $\left(\frac{1}{1-\lambda^{k}}\right)G_{1}(0)\subset
 \Gamma_{G}$ and $(1-\lambda^{k})\Gamma_{G}\subset\Gamma_{G}$.\
 \end{lem}\
\medskip

\begin{proof}Let $k\in\mathbb{Z}^{*}$ such that $\lambda^{k}\neq1$.\
\\
 (i) Let $b\in \Gamma_{G}$ and $g_{1}=(b,\mu)\in G\backslash\mathcal{T}_{1}$, so $g_{1}=h^{n_{1}}\circ
 f^{m_{1}}\circ\dots\circ h^{n_{p}}\circ f^{m_{p}}$ for some $p\in
 \mathbb{N}$ and $n_{1},m_{1},\dots,n_{p},m_{p}\in\mathbb{Z}$. Then
 $$g_{1}(z)=\lambda^{n_{1}}\left(\lambda^{m_{1}}\left(\dots(\lambda^{n_{p}}(\lambda^{m_{p}}(z-a)+a)\dots-a\right)+a\right),\ \ \ z\in\mathbb{C}$$
 It follows that
 $\mu=\lambda^{j}$ with $j=n_{1}+m_{1}+\dots+n_{p}+m_{p}$ and so $\Lambda_{G}=\{\lambda^{j},\ \
 j\in\mathbb{Z}\}$.\ It follows that
 $g=(b,\lambda)=g_{1}^{-j+1}\in G$.\
\\
\\
 (ii)  Let $a\in G_{1}(0)$ and  $g=T_{a}\circ h \circ T_{-a}$, so $g=(a, \lambda)$ and
$g^{k}\circ h^{k}\circ g^{-k}\circ
h^{-k}(z)=z+(1-\lambda^{k})^{2}a$,  $z\in\mathbb{C}$. So $(\lambda^{k}-1)^{2}a\in G_{1}(0)$.\
\
\\
\\
Let $b\in \Gamma_{G}$. By (i),  $g=(b,\lambda)\in G$ then $g^{k}\circ h^{k}\circ g^{-k}\circ
h^{-k}(z)=z+(1-\lambda^{k})^{2}b$,  $z\in\mathbb{C}$. So $(1-\lambda^{k})^{2}b\in G_{1}(0)$.\
 \\
 \\
 (iii) Let $b\in G_{1}(0)$. We have $T_{b}\in G_{1}$ and
\begin{align*}
h^{k}\circ T_{b}\circ h^{-k}(z)& =\lambda^{k}(\lambda^{-k}z+b)\\
\ & =z+\lambda^{k}b,
\end{align*} then
 $\lambda^{k}b\in G_{1}(0)$\
  \\
 \\
 (iv) Let $a\in G_{1}(0)$. We have $T_{a}\in G_{1}$ and
\begin{align*}
T_{a}\circ h^{k}(z)& =\lambda^{k}z+a\\
\ &
=\lambda^{k}\left(z-\frac{a}{1-\lambda^{k}}\right)+\frac{a}{1-\lambda^{k}},
\end{align*} then
$f_{1}=T_{a}\circ
h^{k}=\left(\frac{a}{1-\lambda^{k}},\lambda^{k}\right)$, so
 $$\frac{a}{1-\lambda^{k}}\in \Gamma_{G}\ \ \ \ (1).$$  \
\\
\\
Let $b\in \Gamma_{G}$. By (ii), $(1-\lambda^{k})^{2}b\in G_{1}(0)$ and by (1), $\frac{(1-\lambda^{k})^{2}b}{1-\lambda^{k}}=(1-\lambda^{k})b\in\Gamma_{G}$.\
\end{proof}
\medskip

\begin{lem} \label{L:bb10+} Let $G$ be a subgroup of $\mathcal{H}(1, \mathbb{C})$  with $0\in\Gamma_{G}$. Then:\
\\
  $(i)$ $\Lambda_{G}z+G_{1}(0)\subset G(z)\subset \Lambda_{G}z+G(0)$ for every $z\in \mathbb{C}$.
  \
  \\
   $(ii)$ $(1-\lambda)\Gamma_{G}\cup G_{1}(0)\subset G(0)\subset  G_{1}(0)\cup\Gamma_{G}$.
\end{lem}
   \medskip

\begin{proof} Let $h=\lambda id_{\mathbb{C}^{n}}\in G$ for some $\lambda\in \Lambda_{G}$, since $0\in\Gamma_{G}$ and let $f'\in G$ with
$f'\circ h \neq h\circ f'$, so  $f=f'\circ h \circ f'^{-1}=(a, \lambda)$ for some $a\in\Gamma_{G}$.\
\\
\\
 {\it Proof of (i):} Let $g=(b, \mu)\in G$, so
 $$g(z)=\left\{\begin{array}{c}
                 \mu(z-b)+b=\mu z+(1-\mu)b,\ \ \ \  if\ g\in G\backslash \mathcal{T}_{1}\ \ \\
                 z+b, \ \ \ \ \ \ \ \ \ \ \ \  \ \ \ \ \ \ \ \ \ \ \ \ \ \ \ \ \ \ \ \ \ \  if\ g\in
                 G_{1}\ \ \ \ \
                                \end{array}
 \right. \ \ \ \ \ \ \ \ \ \ \ \ (2)$$
 By (2), $b,(1-\mu)b\in G(0)$, so
 $G(z)\subset\Lambda_{G}z+G(0)$. Conversely, let
 $\mu\in\Lambda_{G}$, $a\in G_{1}(0)$ so $T_{a}\in G$.  By Lemma ~\ref{L:bb10},(iv),
 $a'=\frac{a}{1-\mu}\in \Gamma_{G}$ and by Lemma ~\ref{L:bb10},(i),  $g=(a',\mu)\in G\backslash
 \mathcal{T}_{1}$. Then $g(z)=\mu(z-a')+a'=\mu z+(1-\mu)a'$, thus
 $g(z)=\mu z + a\in G(z)$. It follows that
 $\Lambda_{G}z+G_{1}(0)\subset G(z)$.
 \
 \\
 \\
 {\it Proof of (ii):} Let $b\in G(0)$, so $b=f(0)$, for some $f=(a,\mu)\in G$. By (2), $b=a\in G_{1}(0)$ if $f\in G_{1}$ and $a\in\Gamma_{G}$ if $f\in G\backslash \mathcal{T}_{1}$.
 By Lemma ~\ref{L:bb10},(i), $\mu=\lambda^{k}\neq1$ for some $k\in \mathbb{Z}^{*}$, then $b=(1-\lambda^{k})a$ and by Lemma ~\ref{L:bb10},(iv), $b\in \Gamma_{G}$.
  It follows that $G(0)\subset G_{1}(0)\cup \Gamma_{G}$. \
 \\
Let $b\in \Gamma_{G}$. By Lemma ~\ref{L:bb10},(i),  $g=(b, \lambda)\in G\backslash \mathcal{T}_{1}$, so $g(0)=(1-\lambda)b\in G(0)$. Then
$(1-\lambda)\Gamma_{G}\subset G(0)$. As $G_{1}(0)\subset G(0)$, the results follows.
\end{proof}
\medskip

Notice that the following Lemma is a consequence of Theorems 2.1
and 3.1  given in \cite{mW}, for a closed subgroup of
$\mathbb{R}^{n}$, by identifying $\mathbb{C}^{n}$ to
$\mathbb{R}^{2n}$, we obtain:
\medskip

\begin{lem}\label{L:9--} Let $H$ be a closed subgroup of $\mathbb{C}$. Then:\
\\
(1) If $H$ is discrete then $H=\mathbb{Z}a$ or
$H=\mathbb{Z}a+\mathbb{Z}b$, for some basis $(a,b)$ of $\mathbb{C}$ over $\mathbb{R}$.\
\\
(2) If $H$ is not discrete then there is one of
the following:
\begin{itemize}
  \item [(i)] $H=\mathbb{C}$.\
  \item [(ii)] $H=\mathbb{R}a$, for some $a\in \mathbb{C}$.\
  \item [(iii)] $H=\mathbb{R}a+ \mathbb{Z}b$, for some basis $(a, b)$ of $\mathbb{C}$ over $\mathbb{R}$.
\end{itemize}
\end{lem}
\medskip

\section{{\bf Some results for the case n=1}}

In this section, we study the case when $n=1$ and $G$ is generated by $f=(a, \lambda)$ and $g=(b, \mu)$ for some
$\lambda\in\mathbb{C}\backslash \mathbb{R}$, $\mu\in\mathbb{C}$ and $a,b\in\mathbb{C}^{n}$ with $a\neq b$.
\medskip

\subsection{{\bf Case: $\mathbf{|\lambda|\neq 1}$}}\

\begin{lem}\label{L:9} Let $\lambda\in\mathbb{C}\backslash S^{1}$, $\mu\in \mathbb{C}$ and   $a,b\in \mathbb{C}^{n}$
with $a\neq b$.  If  $G$  is the group generated by
$f=(a, \lambda)$  and  $g=(b,\mu)$  then  $\overline{G_{1}(z)}=\mathbb{C}(b-a)+a$ for every $z\in \mathbb{C}(b-a)+a$.
In particular,  $\overline{G(z)}=\mathbb{C}(b-a)+a$ for every $z\in \mathbb{C}(b-a)+a$.
\end{lem}
\medskip

\begin{proof} We can assume that $\mu=\lambda$, otherwise we replace $g$ by $g\circ f \circ g^{-1}$ and so $G$ will  be the group generated by\ $f=(a, \lambda)$  and
 $g=(b,\lambda)$. Suppose that $|\lambda|>1$ (leaving to replace  $f$ by $f^{-1}$).
\
\\
$(i)$ Firstly, we will show that $G_{1}$ is not discrete.  Denote
by $G'=T_{-a}\circ G \circ T_{a}$, then $G'$ is generated by
$h=T_{-a}\circ f \circ T_{a}$ and $g'=T_{-a}\circ g \circ T_{a}$.
We obtain
 $h=\lambda.id_{\mathbb{C}^{n}}$ and $g'=(b-a, \lambda)$.

Therefore $h^{k}\circ g'^{k}\circ h^{-k}\circ g'^{-k}(z)=z-(\lambda^{k}-1)^{2}(b-a)$,  $z\in\mathbb{C}^{n}$ for every $k\in\mathbb{Z}$.
 Write $T_{a_{k}}=h^{k}\circ g'^{k}\circ h^{-k}\circ g'^{-k}$ is the translation by $a_{k}=- (\lambda^{k}-1)^{2}(b-a)$. One has \begin{align*}
 a_{k}-a_{k+1}& = \left((\lambda^{k+1}-1)^{2}- (\lambda^{k}-1)^{2}\right)(b-a)\\
 \ & =\lambda^{k}(\lambda-1)(\lambda^{k+1}+\lambda^{k}-2)(b-a).
\end{align*}

Since
 $|\lambda|>1$, it follows that $$\underset{k\to-\infty}{lim}\|a_{k}-a_{k+1}\|=0.\ \ \ \ \  \ \ \ \ (1)$$
\\
so  $G'_{1}(0)$ can not be discrete.
\
\\
\\
$(ii)$ Secondly,  suppose that $\overline{G'_{1}(0)}\neq \mathbb{C}(b-a)$, then by (i) and  Lemma~\ref{L:9--}, there are two cases:
\\
$\bullet$ Suppose that $\overline{G'_{1}(0)}=(\mathbb{R}\alpha +\mathbb{Z}\beta)(b-a)$,
 for some basis $(\alpha(b-a),\beta(b-a))$ of $\mathbb{C}(b-a)$ over $\mathbb{R}$. Let $u=\beta(b-a)$ and $T_{u}$
the translation by $u$. See that $u\in\overline{G'_{1}(0)}\subset
\overline{G'(0)}$ and so $G'(u)\subset \overline{G'(0)}$. Remark
that $T_{u}\in \overline{G'_{1}}$, where $\overline{G'_{1}}$ is
the closure of $G'_{1}$ in $\mathcal{T}_{n}$, then
$g_{1}=T_{-u}\circ h \circ T_{u}\in \overline{G'}$, so
$g_{1}=(u,\lambda)$. Let $b_{k}=- (\lambda^{k}-1)^{2}u$ and
$T_{b_{k}}$ be the translation by $b_{k}$, $k\in \mathbb{Z}$. As
above, we have  $T_{b_{k}}=h^{k}\circ g_{1}^{k}\circ h^{-k}\circ
g_{1}^{-k}\in G'_{2}\cap \mathcal{T}_{n}$, since $h$ and $g_{1}\in
G'_{2}$,  for every $k\in \mathbb{Z}$. Therefore, by $(1)$,
$\underset{k\to-\infty}{lim}\|b_{k}-b_{k+1}\|=0$, so
$\|b_{k_{0}}-b_{k_{0}+1}\|<\frac{1}{2}$,
 for some $k_{0}\in \mathbb{Z}$. Let $v=b_{k_{0}}-b_{k_{0}+1}$ then
 $v=\beta\left((\lambda^{k_{0}+1}-1)^{2}-(\lambda^{k_{0}}-1)^{2}\right)(b-a)\in\beta\mathbb{R}(b-a)$, so $v\notin (\alpha\mathbb{R} +\beta\mathbb{Z})(a-b)=\overline{G'_{1}(0)}$ since $\|v\|<\frac{1}{2}$, a contradiction, because
  $v=T_{b_{k_{0}}}\circ T_{b_{k_{0}+1}}(0)\in (G'_{2}\cap \mathcal{T}_{n})(0)\subset \overline{G'_{1}(0)}$.\
  \\
 $\bullet$  Suppose that $\overline{G'_{1}(0)}=\alpha\mathbb{R}(a-b)$, for some $\alpha\in\mathbb{C}^{*}$. As
  $\lambda\in\mathbb{C}\backslash \mathbb{R}$, so $\alpha\mathbb{R}(a-b)$ can not be invariant by $h$. On the other hand,
  $G'_{1}(0)\subset \alpha\mathbb{R}(a-b)$, then for any $T_{v}\in G'_{1}$, one has
   $v\in \alpha\mathbb{R}(a-b)$ , so $T'=h\circ T_{v} \circ h^{-1}=T_{h(v)}\in G'_{1}$, hence $h(v)=\lambda v\in \alpha\mathbb{R}(a-b)$, a contradiction.\
   \\
   \\
$(iii)$ Finally, we conclude that $\overline{G'_{1}(0)}=\mathbb{C}(a-b)$ and  by Lemma ~\ref{L:222}, $\overline{G'(0)}=\mathbb{C}(a-b)$.
It follows that
$\overline{G(a)}=T_{b}(\mathbb{C}(a-b))=\mathbb{C}(b-a)+a$.
\end{proof}
\bigskip

\subsection{{\bf Case: $\mathbf{|\lambda|= 1}$}}\ In this case, write $\lambda=e^{i\theta}$,
$\theta\in\mathbb{R}$. We identify $\mathbb{C}$ to
$\mathbb{R}^{2}$, by the isomorphism $\varphi:
z=x+iy\longrightarrow (x,y)$. State the following results:

\begin{thm}\label{T:001} Let $G$ be a group generated by $h=e^{i\theta}Id_{1}$ and $f=(a, e^{i\theta})$,
$a\in \mathbb{C}^{*}$. Then there is one of the following:\
\\
(i) Every orbit of $G$ is dense in $\mathbb{C}$. In this case $f\notin \mathcal{SR}_{1}$.\
\\
(ii) Every orbit of $G$ is closed and discrete. In this case $f\in \mathcal{SR}_{1}$.
\end{thm}
\bigskip

\begin{prop}\label{p:a10} If $\theta\notin \pi\mathbb{Q}$, $\mu\in \mathbb{C}$ and $a,b\in \mathbb{C}^{n}$, with $a\neq b$. If
 $G$  is the group generated by $f=(a,e^{i\theta})$  and  $g=(b,\mu)$ then $\overline{G(a)}=\mathbb{C}(b-a)+a$.
\end{prop}
\medskip

\begin{proof} By Lemma~\ref{L:222}, $\Delta=\mathbb{C}(b-a)+a$ is $G$-invariant and $G_{/\Delta}$ is a subgroup of $\mathcal{H}(1, \mathbb{C})$.\
\\
First, we can assume that $\mu=e^{i\theta}$, otherwise we replace $g$ by $g\circ f \circ g^{-1}$,
second we suppose that $a=0$, otherwise we replace $G$ by $T_{a}\circ G \circ T_{-a}$.
 Then we will show that $\overline{G(0)}= \mathbb{C}$.\
\\
  Let $G'=\varphi\circ G \circ
\varphi^{-1}$, then $G'$ is the group generated by
$R_{1}=\varphi\circ h \circ \varphi^{-1}$ and $R_{2}=\varphi\circ
f \circ \varphi^{-1}$. By a simple calculus, we can check that
$R_{1}=\left[\begin{array}{cc}
 cos\theta & -sin\theta \\
 sin\theta &cos\theta
 \end{array}\right]$ and $R_{2}=T_{\varphi(b)}\circ R_{1} \circ T_{-\varphi(b)}$ is the rotation with center $\varphi(b)$
 and angle $\theta$. Notice by $\|.\|$ the Euclidean norm on $\mathbb{R}^{2}$ defined by $\|(x,y)\|=\sqrt{x^{2}+y^{2}}$.
  Let $u=(x_{0},y_{0})\in \mathbb{R}^{2}$ and $o=(0,0)$. There are three cases:
\
\\
\\
 (1) Suppose that $u\neq o$. Write the closed ball $D=\{v\in \mathbb{R}^{2},\ \ \|v\|< \|u\|\}$ and its border $C=\{v\in \mathbb{R}^{2},\ \ \|v\|=\|u\|\}$.
\\
(i) Firstly, we will prove that $o\in \mathbb{R}^{2}\backslash
T(D)$ for some $T\in G'_{1}$. For every $z\in\mathbb{C}^{n}$, on
has
  \begin{align*}
  h\circ f\circ h^{-1}\circ f^{-1}(z)& = e^{i\theta}(e^{-i\theta}[e^{i\theta}(e^{-i\theta}(z-a)+a)-a]+a)\\
  \ & = z+\left(1-e^{2i\theta}\right)a.
  \end{align*}
Write $c=\left(1-e^{2i\theta}\right)a$, hence $h\circ f\circ h^{-1}\circ f^{-1}=T_{c}\in G\backslash\{id_{\mathbb{C}}\}$ since $\theta\notin \pi\mathbb{Q}$.\
 Then  $T_{\varphi(c)}=\varphi\circ T_{c} \circ \varphi^{-1}\in G'$, so
 $T_{n\varphi(c)}(o)\in \mathbb{R}^{2}\backslash D$, for some $n\in\mathbb{N}$, we take $T=T_{n\varphi(c)}$.
 \\
 \\
(ii) Secondly, let's prove that $T(D)\subset \overline{G'(u)}$. Let
$b\in T(D)$  and set $C_{b}=\{v\in \mathbb{R}^{2},\ \ \|v\|=\|b\|\}$. By (i), $o\notin T(D)$ then $C_{b}\cap T(C)\neq \emptyset$.
 Let $b'\in C_{b}\cap T(C)$, therefore $b'\in \overline{G'(u)}$, because $T\in G'$ and the orbit of $u$ by $R_{1}$ is dense in $C$,  since
 $\theta\notin \pi\mathbb{Q}$, so $C\subset \overline{G'(u)}$. In the same way,  one has $C_{b}\subset \overline{G'(b')}\subset \overline{G'(u)}$, by $R_{1}$.
   It follows that $b\in C_{b}\subset \overline{G'(u)}$ and so $T(D)\subset \overline{G'(u)}$.\
   \\
   \\
(iii) Finally, we conclude that $G'(u)$ is locally dense for every
$u\neq o$. \
\\
\\
(2) Suppose that $u=o$, so $R_{2}(o)\neq o$, by applying $(1)$ on
$v=R_{2}(o)$,
 we obtain $G'(v)$ is locally dense, so $G'(o)$ is locally dense, since $G'(o)=G'(v)$.
 \\
 \\
 (3) We conclude that every orbit of $G'$ is dense in $\mathbb{R}^{2}$, since $\mathbb{R}^{2}$
 is connected and every orbit is locally dense. It follows that, every orbit of $G$ is dense in $\mathbb{C}$.
\end{proof}
\medskip

\begin{lem}\label{L:b10} Let $\theta\in \pi(\mathbb{Q}\backslash\mathbb{Z})$, $\mu\in \mathbb{C}$ and $a,b\in \mathbb{C}^{n}$, with $a\neq b$. If
 $G$  is the group generated by $f=(a,e^{i\theta})$  and  $g=(b,\mu)$ then  every orbit of $G$ is dense in
 $\mathbb{C}(b-a)+a$ or is closed and discrete.
\end{lem}
\medskip

\begin{proof} By Lemma~\ref{L:222}, $\Delta=\mathbb{C}(b-a)+a$ is $G$-invariant and $G_{/\Delta}$ is a subgroup of $\mathcal{H}(1, \mathbb{C})$.
  Firstly, we can assume that $\mu=e^{i\theta}$, otherwise we replace $g$ by $g\circ f \circ g^{-1}$, secondly we suppose that $a=0$, otherwise we replace $G$ by $T_{a}\circ G \circ T_{-a}$.
 Then we will show every orbit of $G$ is dense in
 $\mathbb{C}$ or closed and discrete.\
 \\
Thirdly, we will show that $G_{1}(0)$ is dense in $\mathbb{C}$ or it is closed discrete. Suppose that
$\overline{G_{1}(0)}\neq\mathbb{C}$ and $G_{1}(0)$ is not discrete. Then by Lemma~\ref{L:9--},
$\overline{G_{1}(0)}= \mathbb{Z}a_{1}+\mathbb{R}a_{2}$ for some $a_{1},a_{2}\in \mathbb{R}$ with $a_{2}\neq0$.
 So $T_{a_{1}}\in\overline{G_{1}}$
 where $\overline{G}$ be the closure of $G$ in
$\mathcal{T}_{n}$. Let
$g=T_{a_{1}}\circ h\circ T_{-a_{1}}$, then $g=(a_{1},e^{i\theta})$.
Since $\theta\in \pi\left(\mathbb{Q}\backslash\mathbb{Z}\right)$, so
$e^{i\theta}\mathbb{R}a_{2}\neq \mathbb{R}a_{2}$. By Lemma
~\ref{L:bb10}.(iii), $e^{i\theta}\mathbb{R}a_{2}\subset
\overline{G_{1}(0)}=\mathbb{Z}a_{1}+\mathbb{R}a_{2}$, a
contradiction. \
 \\
 \\
 We  conclude that $G_{1}(0)$ is dense in $\mathbb{C}$ or
 closed and discrete.\
 \\
Finally, by Lemma~\ref{L:bb10}, (iii), (iv) and Lemma ~\ref{L:bb10+},(i) and (ii) we have $G_{1}(0)$ is closed discrete or dense
if and only if are $\Gamma_{G}$ and  $G_{1}(0)$ and this is equivalent to is $G(0)$.
 On the other hand, $\theta\in \pi(\mathbb{Q}\backslash\mathbb{Z})$, so by  Lemma~\ref{L:bb10}.(i),
 $\Lambda_{G}$ is finite and the proof results from Lemma~\ref{L:bb10+}, (i) and (ii).
\end{proof}
\bigskip

\begin{prop}\label{p:b10} Let $\theta\in \pi\left(\mathbb{Q}\backslash\mathbb{Z}\right)$, $\mu\in \mathbb{C}$ and $a,b\in \mathbb{C}^{n}$, with $a\neq b$. If
 $G$  is the group generated by $f=(a,e^{i\theta})$  and  $g=(b,\mu)$ then $G(a)$ is closed and discrete if and only
 if $G\subset \mathcal{S}_{2}\mathcal{R}_{n}$ or $G\subset\mathcal{S}_{3}\mathcal{R}_{n}$.
 \end{prop}
\medskip

To prove Proposition ~\ref{p:b10}, we need to introduce the following Lemmas:\

\begin{lem}\label{L:b101} Let $G$  be the group generated by $h=e^{i\theta}Id_{\mathbb{C}}$  and  $f=(a_{0},e^{i\theta})$
 with  $a_{0}\in \mathbb{C}^{*}$ and $\theta\in \mathbb{R}$. If $G_{1}(0)=\mathbb{Z}a_{1}+\mathbb{Z}a_{2}$ where $(a_{1},a_{2})$ is a basis
 of $\mathbb{C}$ over $\mathbb{R}$  then there exists $P\in GL(2, \mathbb{C})$ such that $Pe_{1}=a_{1}$, $Pe_{2}=a_{2}$ and  $P^{-1}R_{\theta}P\in SL(2,
 \mathbb{Z})$, where $R_{\theta}=\left[\begin{array}{cc}
                                         cos\theta & -sin\theta \\
                                         sin\theta & cos\theta
                                       \end{array} \right]$ where $(e_{1},e_{2})$ is the canonical basis of $\mathbb{R}^{2}$.
\end{lem}
\medskip

\begin{proof} If  $G_{1}(0)=\mathbb{Z}a_{1}+\mathbb{Z}a_{2}$ with
$(a_{1},a_{2})$ is a basis of $\mathbb{C}$ over $\mathbb{R}$. By
Lemma~\ref{L:bb10}.(iii), $e^{i\theta}a_{1},
e^{i\theta}a_{2}\in G_{1}(0)$, so
$$\left\{\begin{array}{c}
 e^{i\theta}a_{1}=na_{1}+ma_{2}\ \ \ \ \  \mathrm{and }\\
e^{i\theta}a_{2}=n'a_{1}+m'a_{2} \ \ \ \ \ \ \ \
                                      \end{array}
\right.$$
 for some $n,m,n',m'\in\mathbb{Z}$. Write $a_{1}=a+ic$ and
$a_{2}=b+id$, $a,c,b,d\in\mathbb{R}$ then:
$$\left\{\begin{array}{c}
(cos\theta+i.sin\theta)(a+ic)=n(a+ic)+m(b+id)\ \ \\
(cos\theta+i.sin\theta)(b+id)=n'(a+ic)+m'(b+id)
\end{array}
\right.$$
So
$$\left\{\begin{array}{c}
a.cos\theta-c.sin\theta=na+mb \\
a.sin\theta+c.cos\theta=nc+md
\end{array}
\right.\ \ \mathrm{and}\ \ \left\{\begin{array}{c}
b.cos\theta-d.sin\theta=n'a+m'b \\
b.sin\theta+d.cos\theta=n'c+m'd
\end{array}
\right.\ \ \ \ (1)$$
\
\\
Write $P=\left[\begin{array}{cc}
                 a & b \\
                 c & d
               \end{array}
\right]$, then (1) is equivalent to

$$R_{\theta}[a,c]^{T}=P[n,m]^{T}\ \ and \ \
R_{\theta}[b,d]^{T}=P[n',m']^{T}.$$
As $Pe_{1}=[a,c]^{T}$ and $Pe_{2}=[b,d]^{T}$,  so
$$R_{\theta}Pe_{1}=P[n,m]^{T}\ \ and \ \
R_{\theta}Pe_{2}=P[n',m']^{T}.$$ As $(a_{1}, a_{2})$ is a basis of
$\mathbb{C}$ over $\mathbb{R}$ one has $P\in GL(2, \mathbb{R})$ and
so
$$P^{-1}R_{\theta}Pe_{1}\in \mathbb{Z}^{2}\ \ and \ \
P^{-1}R_{\theta}Pe_{2}\in \mathbb{Z}^{2}.$$ It follows that
$P^{-1}R_{\theta}P\in SL(2, \mathbb{Z})$.
\end{proof}
\bigskip

\begin{lem}\label{L:1b} Let $G$  be the subgroup of $\mathcal{H}(1,\mathbb{C})$ generated by $h=e^{i\theta}Id_{\mathbb{C}}$  and  $f=(a,e^{i\theta})$
 with  $a\in \mathbb{C}^{*}$ and $\theta\in H_{2}\cup H_{3}$. Then
 $$\left(\mathbb{Z}(1-e^{-i\theta})^{2}+\mathbb{Z}(1-e^{i\theta})^{2}\right)a\subset G(0)\subset\left(\mathbb{Z}(1-e^{-i\theta})+\mathbb{Z}(1-e^{i\theta})\right)a.$$
 \end{lem}
\medskip

\begin{proof} Denote  by $a_{1}=(1-e^{-i\theta})a$ and  $a_{2}=(1-e^{i\theta})a$.\
\\
$\bullet$ Firstly, we will prove that $\mathbb{Z}a_{1}+\mathbb{Z}a_{2}$ is $G$-invariant:\
\\
-  If $\theta\in H_{2}$, suppose that $\theta=\frac{\pi}{2}$. Then $e^{-i\theta}=-i$, $e^{i\theta}=i$, so
 $a_{1}=(1-i)a$ and $a_{2}=(1+i)a$. Let $u=na_{1}+ma_{2}$, for some $n, m\in \mathbb{Z}$, so \begin{align*}
 h(u)& =i(n(1-i)a+m(1+i))a\\
 \ & =na_{2}-ma_{1}
 \end{align*}
 and
  \begin{align*}
 f(u)& =i\left(n(1-i)a+m(1+i)a-a\right)+a\\
 \ & = n(1+i)a-m(i-1)a+(1-i)a\\
 \ & =na_{2}-(m-1)a_{1}
 \end{align*}
 Then $h(u), f(u)\in \mathbb{Z}a_{1}+\mathbb{Z}a_{2}$. It follows that $\mathbb{Z}a_{1}+\mathbb{Z}a_{2}$ is $G$-invariant.\
 \\
 \\
- If $\theta\in H_{3}$, suppose that $\theta=\frac{3\pi}{2}$. Then $e^{-i\theta}=e^{-i\frac{\pi}{3}}$ and  $e^{i\theta}=e^{i\frac{\pi}{3}}$. As
 $1-e^{i\frac{\pi}{3}}=e^{-i\frac{\pi}{3}}$ and  $1-e^{-i\frac{\pi}{3}}=e^{i\frac{\pi}{3}}$ so
 $a_{1}=e^{i\frac{\pi}{3}}a$ and $a_{2}=e^{-i\frac{\pi}{3}}a$. Let $u=na_{1}+ma_{2}$, for some $n, m\in \mathbb{Z}$, so \begin{align*}
 h(u)& =e^{\frac{i\pi}{3}}(ne^{i\frac{\pi}{3}}a+me^{-\frac{i\pi}{3}}a)\\
 \ & =ne^{\frac{2i\pi}{3}}a-ma\\
 \ & =- ne^{\frac{-i\pi}{3}}a-m(e^{i\frac{\pi}{3}}+e^{-\frac{i\pi}{3}})a\\
\ & =(-n-m)e^{\frac{-i\pi}{3}}a-me^{i\frac{\pi}{3}}a
 \end{align*}
 and
  \begin{align*}
 f(u)& =e^{\frac{i\pi}{3}}\left((ne^{i\frac{\pi}{3}}a+me^{-\frac{i\pi}{3}}a-a\right)+a\\
 \ & =ne^{\frac{2i\pi}{3}}a-(m-1)a-e^{\frac{i\pi}{3}}a\\
 \ & =- ne^{\frac{-i\pi}{3}}a-(m-1)(e^{\frac{i\pi}{3}}+e^{-\frac{i\pi}{3}})a-e^{\frac{i\pi}{3}}a\\
\ & =(-n-m+1)e^{\frac{-i\pi}{3}}a-me^{\frac{i\pi}{3}}a
 \end{align*}
 Then $h(u), f(u)\in \mathbb{Z}a_{1}+\mathbb{Z}a_{2}$. It follows that $\mathbb{Z}a_{1}+\mathbb{Z}a_{2}$ is $G$-invariant.\
\
\\
$\bullet$ Secondly, $G(0)\subset \mathbb{Z}a_{1}+\mathbb{Z}a_{2}$, since $0\in\mathbb{Z}a_{1}+\mathbb{Z}a_{2}$ and by above,  $\mathbb{Z}a_{1}+\mathbb{Z}a_{2}$
 is $G$-invariant. In particular $$G_{1}(0)\subset G(0)\subset \mathbb{Z}a_{1}+\mathbb{Z}a_{2}\ \ \ \ (1).$$\
 \\
 \\
 $\bullet$ Finally, by Lemma ~\ref{L:bb10}.(iii), $(1-e^{-i\theta})^{2}a, (1-e^{i\theta})^{2}a\in G_{1}(0)$ since $a\in\Gamma_{G}$. As $G_{1}(0)$ is an
 additive group then $$\mathbb{Z}(1-e^{-i\theta})^{2}a+\mathbb{Z}(1-e^{i\theta})^{2}a\subset G_{1}(0)\subset G(0)\ \ \ \ (2).$$\
 \end{proof}
\bigskip

\begin{lem}\label{L:LL1} Let $G$  be the subgroup of $\mathcal{H}(1,\mathbb{C})$ generated by $h=\lambda Id_{\mathbb{C}}$  and  $f=(a,\lambda)$
 with  $a\in \mathbb{C}^{*}$, $\lambda\notin \mathbb{R}$. If $G_{1}(0)$ is discrete then $G_{1}(0)=\mathbb{Z}a_{1}+\mathbb{Z}a_{2}$
 for some basis $(a_{1}, a_{2})$ of $\mathbb{C}$ over $\mathbb{R}$.
\end{lem}
\medskip

\begin{proof} By Lemma ~\ref{L:bb10}.(ii), $0\neq(\lambda-1)^{2}a\in G_{1}(0)$. Write $a_{1}=(\lambda-1)^{2}a$. By Lemma ~\ref{L:bb10}.(iii),
$a_{2}=\lambda a_{1}\in G_{1}(0)$. As $\lambda\notin \mathbb{R}$, $(a_{1}, a_{2})$ is a basis of $\mathbb{C}$ over
$\mathbb{R}$. Then $\mathbb{Z}a_{1}+\mathbb{Z}a_{2}\subset G_{1}(0)$ since $G_{1}(0)$ is an additive group.
By Lemma ~\ref{L:9--}, $G_{1}(0)=\mathbb{Z}a'_{1}+\mathbb{Z}a'_{2}$ for some basis $(a'_{1}, a'_{2})$ of $\mathbb{C}$ over
$\mathbb{R}$.
\end{proof}
\medskip

\begin{proof}[Proof of Proposition~\ref{p:b10}] By Lemma~\ref{L:222}, $\Delta=\mathbb{C}(b-a)+a$ is $G$-invariant and $G_{/\Delta}$ is a subgroup of $\mathcal{H}(1, \mathbb{C})$.\
\\
First, we can assume that $\mu=e^{i\pi\theta}$, otherwise we replace $g$ by $g\circ f \circ g^{-1}$, second we suppose that $a=0$,
leaving to replace $G$ by $T_{a}\circ G \circ T_{-a}$.
 Then we will show that $G(0)$ is closed and discrete if and only if $G\subset\mathcal{S}_{2}\mathcal{R}_{n}$ or
  $G\subset\mathcal{S}_{3}\mathcal{R}_{n}$. Then $G$  is generated by $h=e^{i\theta}Id_{\mathbb{C}}$  and  $g=(b,e^{i\theta})$
 with  $b\in \mathbb{C}^{*}$ and $\theta\in \mathbb{R}$. If
 $G(0)$ is discrete so is $G_{1}(0)$. Therefore, by Lemma ~\ref{L:LL1}, $G_{1}(0)=\mathbb{Z}a_{1}+\mathbb{Z}a_{2}$ for some basis $(a_{1},a_{2})$
  of $\mathbb{C}$ over $\mathbb{R}$.
 By Lemma~\ref{L:b101}, there exists $P\in GL(2, \mathbb{R})$ such that $P^{-1}R_{\theta}P\in SL(2,
 \mathbb{Z})$, where $R_{\theta}=\left[\begin{array}{cc}
                                         cos\theta & -sin\theta \\
                                         sin\theta & cos\theta
                                       \end{array}
 \right]$. Write $P=\left[\begin{array}{cc}
                                         a' & b' \\
                                         c' & d'
                                       \end{array}
 \right]$  and $A=P^{-1}R_{\theta}P$, so $$A=\left[\begin{array}{cc}
                                         cos\theta-\left(\frac{a'b'+d'c'}{a'd'-b'c'}\right)sin\theta & \left(\frac{b^{2}+d^{2}}{a'd'-b'c'}\right)sin\theta \\
                                         \\
                                         -\left(\frac{a'^{2}+c'^{2}}{a'd'-b'c'}\right)sin\theta & cos\theta+\left(\frac{a'b'+d'c'}{a'd'-b'c'}\right)sin\theta
                                       \end{array}
 \right].$$ As $A\in SL(2,\mathbb{Z})$ then there exist $n,m\in\mathbb{Z}$ such that $$\left\{\begin{array}{c}
                                                       cos\theta-\left(\frac{a'b'+d'c'}{a'd'-b'c'}\right)sin\theta =n\\
                                                       cos\theta+\left(\frac{a'b'+d'c'}{a'd'-b'c'}\right)sin\theta=m
                                                     \end{array}
 \right.$$ so $cos\theta=\frac{n+m}{2}\in\frac{1}{2}\mathbb{Z}$.
 Hence $cos\theta\in\left\{-\frac{1}{2},0,\frac{1}{2}\right\}$ since
 $\theta\notin\pi\mathbb{Z}$, therefore
 $sin\theta\in\left\{-1,-\frac{\sqrt{3}}{2},1,\frac{\sqrt{3}}{2}\right\}$.
 Thus
 $\theta\in(\frac{\pi}{2}+\pi\mathbb{Z})\cup(-\frac{\pi}{3}+\pi\mathbb{Z})\cup(\frac{\pi}{3}+\pi\mathbb{Z})$. Then
 $G\subset\mathcal{S}_{2}\mathcal{R}_{n}$ if $\theta\in(\frac{\pi}{2}+\pi\mathbb{Z})$ and  $G\subset\mathcal{S}_{3}\mathcal{R}_{n}$
  if $\theta\in(-\frac{\pi}{3}+\pi\mathbb{Z})\cup(\frac{\pi}{3}+\pi\mathbb{Z})$. The converse follows from
 Lemma~\ref{L:b10}. The proof is complete.
 \end{proof}
\bigskip

\subsection{{\bf Proof of Theorem~\ref{T:001}}}\

\begin{lem}\label{L:2} $\lambda,\mu\in \mathbb{C}^{*}$ and $a,b\in \mathbb{C}^{n}$, with $a\neq b$. If
 $G$  is the group generated by $f=(a,\lambda)$  and  $g=(b,\mu)$  such that $\overline{G(a)}=\mathbb{C}(b-a)+a$ then
 $\overline{G(z)}=\mathbb{C}(b-a)+a$, for every $z\in\mathbb{C}(b-a)+a$.
\end{lem}
\medskip

\begin{proof} Let $z\in\mathbb{C}(b-a)+a$. There are three cases:\
\\
- If $|\lambda|\neq 1$ or $|\mu|\neq 1$, then by Lemma  ~\ref{L:9}, $\overline{G_{1}(z)}=\mathbb{C}(b-a)+a$, so $\overline{G(z)}=\mathbb{C}(b-a)+a$.
\
\\
- If $|\lambda|=|\mu|= 1$, then $G\subset \mathcal{R}_{n}$. Since $z\in \overline{G(a)}$,
then there exists a sequence $(g_{m})_{m}\subset G$ such that $\underset{m\to +\infty}{lim}g_{m}(a)=z$. Write $g_{m}=(a_{m},\varepsilon_{m})$,
with $|\varepsilon_{m}|=1$ for every $m\in\mathbb{N}$. Since $(g_{m}a)_{m}$ is bounded then is $(a_{m})_{m}$. Therefore, there is a subsequence
 $(a_{\varphi(m)})_{m}$  such that $\underset{m\to +\infty}{lim}a_{\varphi(m)}=c$ and $\underset{m\to +\infty}{lim}\varepsilon_{\varphi(m)}=\varepsilon$,
 for some $c\in \mathbb{C}^{n}$ and  $\varepsilon\in S^{1}$. Moreover, $\underset{m\to +\infty}{lim}g_{\varphi(m)}=h=(c,\varepsilon)$. Therefore
 $\underset{m\to +\infty}{lim}g_{\varphi(m)}(a)=h(a)=z.$ So $\underset{m\to +\infty}{lim}g^{-1}_{\varphi(m)}(z)=h^{-1}(z)=a$. It follows that
  $a\in \overline{G(z)}$, hence $\mathbb{C}(b-a)+a=\overline{G(a)}\subset \overline{G(z)}$.
\end{proof}
\medskip

\begin{cor}\label{CC:01} Let $\theta\in \mathbb{R}$, $\mu\in \mathbb{C}$ and $a,b\in \mathbb{C}^{n}$, with $a\neq b$. If
 $G$  is the group generated by $f=(a,e^{i\theta})$  and  $g=(b,\mu)$  such that
 $G\backslash\mathcal{SR}_{n}\neq\emptyset$ then every orbit of $G$ is dense in $\mathbb{C}(b-a)+a$.
\end{cor}
\medskip

\begin{proof} The proof results from Lemma~\ref{L:9}, Proposition ~\ref{p:a10}, Lemma~\ref{L:b10}, Proposition~\ref{p:b10} and Lemma~\ref{L:2}.
\end{proof}
\medskip

\begin{lem} \label{L:3} Let  $G$ be a non abelian subgroup of
$\mathcal{H}(n,\mathbb{C})$. If $G\backslash \mathcal{R}_{n}\neq \emptyset$, then for every $z\in\mathbb{C}^{n}$
 we have $\Gamma_{G}\subset \overline{G(z)}$.
\end{lem}
\medskip

\begin{proof} Let  $z\in\mathbb{C}^{n}$,
$a\in\Gamma_{G}$ and $f=(a,\lambda)\in G\backslash \mathcal{R}_{n}$ with  $|\lambda|\neq 1$. Suppose that
$|\lambda|>1$ and so
$$\underset{k\longrightarrow-\infty}{lim}f^{k}(z)=\underset{k\longrightarrow-\infty}{lim}\lambda^{k}(z-a)+a=a.$$
\ Hence  $a\in \overline{G(z)}$. It follows that
$\Gamma_{G}\subset \overline{G(z)}$.
\end{proof}
\
\\
{\it Proof of Theorem ~\ref{T:001}:} The proof of Theorem ~\ref{T:001} results from Proposition ~\ref{p:b10} and Corollary ~\ref{CC:01}.
\
\\
\\
{\it Proof of Theorem ~\ref{TT:1}:} Let $\widetilde{G}=\varphi^{-1}\circ G \circ \varphi$, so $\widetilde{G}$ is a non abelian subgroup of
$\mathcal{H}(1, \mathbb{C})$. Firstly, if $\left(H_{2}\cup H_{3}\right)\backslash\{\theta,\theta'\}\neq\emptyset$
then $\widetilde{G}\backslash \mathcal{SR}_{2}\neq\emptyset$. Therefore:\
\\
The proof of (1).(i) results from Theorem ~\ref{T:001}. Let's prove (1).(ii):\
\\
Suppose that  $\theta\in H_{2}$ and $\theta'\in H_{3}$, then by using the analytic form $f=\varphi^{-1}\circ R_{\theta}\circ \varphi$ (resp. $g=\varphi^{-1}\circ R_{\theta'}\circ \varphi$)
 of $R_{\theta}$ (resp. $R_{\theta'}$)
 we have $f=(a,e^{i\theta})\in \mathcal{S}_{2}\mathcal{R}_{2}$ and
$g=(b,e^{i\theta'})\in \mathcal{S}_{3}\mathcal{R}_{2}$, where $\varphi(a)$ (resp.  $\varphi(b)$) is the center of $R_{\theta}$ (resp. $R_{\theta'}$). Then
$f\circ g=\left(c,e^{i(\theta+\theta')}\right)$ with $c=\frac{e^{i\theta}(b-a)-a}{1-e^{i(\theta+\theta')}}$. See that
$\theta+\theta'\in\left(\frac{5\pi}{6}+\pi\mathbb{Z}\right)\cup \left(\frac{7\pi}{6}+\pi\mathbb{Z}\right)$. Then
$\theta+\theta'\notin H_{2}\cup H_{2}$. The assertion $(1).(ii)$
follows then from (1).(i).\
\\
(2) In this case, we can assume that $\widetilde{G}\subset \mathcal{SR}_{2}$. By Lemma ~\ref{L:1b}, if $\widetilde{G}\subset \mathcal{S}_{2}\mathcal{R}_{2}$ or
 $\widetilde{G}\subset \mathcal{S}_{3}\mathcal{R}_{2}$ then every orbit is closed and discrete. By (1), it remains to verify the following case:
  $\theta,\theta'\in H_{i}$  for some $i\in\{2,3\}$. Then $\widetilde{G}\subset \mathcal{S}_{i}\mathcal{R}_{2}$.
  The results follows from Lemma ~\ref{L:1b}. The proof is complete.
\bigskip

\section{\textbf{Some results in the case  $G\backslash \mathcal{SR}_{n}\neq\emptyset$ for $n\geq 1$}}
\medskip

We give some Lemmas and propositions, will be
used to prove Theorem ~\ref{T:1}.

\medskip

\begin{lem}\label{L:8} Let $G$ be a non abelian subgroup of $\mathcal{H}(n, \mathbb{C})$ such that
$G\backslash \mathcal{SR}_{n}\neq\emptyset$. So $G_{1}\neq\{id_{\mathbb{C}^{n}}\}$ and if $0\in \Gamma_{G}$ then $G_{1}(0)\subset\Gamma_{G}$.
\end{lem}
\medskip

\begin{proof} Let $f, g\in G$ such that
 $f\circ g\neq g\circ f$. Write $f:z\longmapsto \lambda z +a$ and $g:z\longmapsto \mu z +b$.
  So for every $z\in\mathbb{C}^{n}$, one has
  \begin{align*}
  f\circ g\circ f^{-1}\circ g^{-1}(z)& =\lambda\left(\mu\left(\frac{1}{\lambda}\left(\frac{1}{\mu}z-\frac{b}{\mu}\right)-\frac{a}{\lambda}\right)+b\right)+a\\
  \ & = z+\left(\lambda-1\right)b+\left(1-\mu\right)a.
  \end{align*}
Hence $f\circ g\circ f^{-1}\circ g^{-1}=T_{c}\in G_{1}\backslash\{id_{\mathbb{C}^{n}}\}$, with $c=\left(\lambda-1\right)b+\left(1-\mu\right)a$.\
\\
Suppose now that $0\in \Gamma_{G}$, so there  $h=\lambda id_{\mathbb{C}^{n}}\in G\backslash \mathcal{SR}_{n}$ for some $\lambda\in\Lambda_{G}$.
Let $a\in G_{1}(0)$, then $T_{a}\circ h \circ T_{-a}=(a,\lambda)\in G\backslash \mathcal{SR}_{n}$. So $a\in \Gamma_{G}$. The proof is complete.
\end{proof}
\bigskip

\begin{prop}\label{p:1} Let  $G$\ be a non abelian subgroup of
$\mathcal{H}(n,\mathbb{C})$  such that $\Lambda_{G}\backslash\mathbb{R}\neq\emptyset$ and
$G\backslash\mathcal{SR}_{n}\neq\emptyset$.  Then $\overline{G(z)}=E_{G}$, for every $z\in E_{G}$.\
\end{prop}
\medskip

To prove the Proposition, we need the following Lemmas:

\begin{lem}\label{L:11} Let  $G$ be a non abelian subgroup of
$\mathcal{H}(n,\mathbb{C})$, $f\in G$ and  $u,v\in\mathbb{C}^{n}$  then
$f(\mathbb{C}u+v)=\mathbb{C}u+f(v)$.\
\end{lem}
\medskip

\begin{proof}  Every  $f\in G$  has the form
$f(z)=\lambda z+a$, $z\in \mathbb{C}^{n}$. Let $\alpha\in
\mathbb{C}$  then $f(\alpha u+v)=\lambda(\alpha
u+v)+a=\lambda\alpha u+(\lambda u+v)=\lambda\alpha u+f(v)$. So $f(\mathbb{C}u+v)\subset\mathbb{C}u+f(v)$, then
$f(\mathbb{C}u+v)=\mathbb{C}u+f(v)$.\
\end{proof}
\medskip

\begin{lem}\label{L:12} Let  $G$ be a non abelian subgroup of
$\mathcal{H}(n,\mathbb{C})$  such that $E_{G}$ is a vector
subspace of $\mathbb{C}^{n}$ and $\Gamma_{G}\neq\emptyset$.  Let \ $a, a_{1},\dots,a_{p}\in
\Gamma_{G}$  such that  $(a_{1},\dots,a_{p})$ and $(a_{1}-a,\dots,a_{p}-a)$ are two basis of
$E_{G}$ and let $D_{k}=\mathbb{C}(a_{k}-a)+a$,
$1\leq k \leq p$.  If \ $D_{k}\subset\overline{G(a)}$  for every
 $1\leq k \leq p$,  then  $\overline{G(a)}=E_{G}$.
\end{lem}
\medskip

\begin{proof} The proof is done by induction on  $\mathrm{dim}(E_{G})=p\geq 1$.
\\
$\bullet$ For $p=1$,   if there exists $a,a_{1}\in\Gamma_{G}$ with $a\neq a_{1}$ such that $D_{1}\subset \overline{G(a)}$, where
$D_{1}=\mathbb{C}(a_{1}-a)+a$,  then $\overline{G(a)}=E_{G}$, since $D_{1}=E_{G}=\mathbb{C}$.
\
\\
$\bullet$ Suppose that Lemma ~\ref{L:12} is true until dimension $p-1$. Let  $G$ be
a non abelian subgroup of \ $\mathcal{H}(n,\mathbb{C})$ with $\Gamma_{G}\neq\emptyset$ and let
$a, a_{1},\dots,a_{p}\in \Gamma_{G}$  such that
$(a_{1},\dots,a_{p})$  is a basis of  $E_{G}$.  Suppose that
$D_{k}\subset\overline{G(a)}$  for every  $1\leq k \leq p$.
\\
Denote by $H$  the vector
subspace of  $E_{G}$ generated by
$(a_{1}-a),\dots,(a_{p-1}-a)$ and $\Delta_{p-1}=T_{a}(H)$. We have  $\Delta_{p-1}=Aff(a,a_{1},\dots, a_{p-1})$.
\\
Set $\lambda,\lambda_{k}\in \Gamma_{G}$,   $1\leq k
\leq p-1$   such that $f=(a,\lambda), f_{k}=(a_{k},\lambda_{k})\in G\backslash\mathcal{SR}_{n}$.  We let
$G_{k}$ be the group generated by $f$ and $f_{k}$ for every $1\leq k \leq
p-1$, so $G_{k}\backslash\mathcal{SR}_{n}\neq\emptyset$. By Corollary ~\ref{CC:01},
we have $\overline{G_{k}(a)}=D_{k}$ for every $k=1,\dots, p-1$. Let  $G'$
be the subgroup of $G$ generated by $f$, $f_{1}$,\dots,$f_{p-1}$,   then
$D_{k}\subset \overline{G'(a)}$ for every $1\leq k \leq p-1$.
\\
By Lemma ~\ref{L:1111} we have $E_{G'}=\Delta_{p-1}$. Let
$G''=T_{-a}\circ G'\circ T_{a}$,   by Lemma ~\ref{L:1}.(iii) we have
$E_{G''}=T_{-a}(\Delta_{p-1})=H$  and
$D'_{k}=T_{-a}(D_{k})\subset \overline{G''(0)}$  for every $1\leq
k \leq p-1$. By induction hypothesis applied to $G''$ we have
$\overline{G''(0)}=H$  so  $\overline{G'(a)}=\Delta_{p-1}$.
Since $G'(a)\subset G(a)$, then $$\Delta_{p-1}\subset \overline{G(a)}\ \ \ \ \ \ \ \ \ \
(1)$$\
\\
Let  $z\in E_{G}\backslash \Delta_{p-1}$  and
$D=\mathbb{C}(a_{p}-a)+z$. Since $(a_{1}-a,\dots,a_{p}-a)$ is a basis of $E_{G}$, so
$H\oplus\mathbb{C}(a_{p}-a)=E_{G}$. As $a,z\in E_{G}$,
then $z-a=x+\alpha(a_{p}-a)$ for some $x\in H$ and
$\alpha\in\mathbb{C}$. Let $y=x+a$, as  $T_{a}(H)=\Delta_{p-1}$ we
have  $y\in\Delta_{p-1}$, and
\begin{align*}
y& =x+a\\
\ & =z-a-\alpha(a_{p}-a)+a\\
\ & =-\alpha(a_{p}-a)+z\in D.
\end{align*}

Hence \ \ $y\in \Delta_{p-1}\cap D.$
\
\\
\\
 By (1) we have $y\in
\overline{G(a)}$.  Then there exists a sequence
$(f_{m})_{m\in\mathbb{N}}$  in  $G$  such that \
$\underset{m\longrightarrow+\infty}{lim}f_{m}(a)=y.$ For every
$m\in\mathbb{N}$  denote by  $f_{m}=(b_{m}, \lambda_{m})$.
\
\\
\\
Remark that $D=\mathbb{C}(a_{p}-a)+y$, since $z,y\in D$. By Lemma ~\ref{L:11} we have
$f_{m}(D_{p})=f_{m}(\mathbb{C}(a_{p}-a)+a)=\mathbb{C}(a_{p}-a)+f_{m}(a)$.
Since  $\underset{m\longrightarrow+\infty}{lim}f_{m}(a)=y$  then for every $v=\alpha(a_{p}-a)+y\in D$, $\alpha\in \mathbb{C}$, one has\
$$\underset{m\longrightarrow+\infty}{lim}f_{m}(\alpha(a_{p}-a)+a)=v.$$\
\\
As $\alpha(a_{p}-a)+a\in D_{p}\subset \overline{G(a)}$,  then  $v\in
\overline{G(a)}$. Therefore $D\subset
\overline{G(a)}$, so  $z\in \overline{G(a)}$, hence
$$E_{G}\backslash \Delta_{p-1}\subset \overline{G(a)}\ \ \ \ \ \ \ \  \ \ \ (2).$$ By
$(1)$ and $(2)$ we obtain  $E_{G}\subset \overline{G(a)}$. Since $\Gamma_{G}\neq\emptyset$ then by Lemma
~\ref{L:020}.(ii), we have $E_{G}$  is $G$-invariant, so $G(a)\subset E_{G}$ since $a\in
E_{G}$. It follows that
$\overline{G(a)}=E_{G}$.
\end{proof}
\bigskip

\begin{proof}[Proof of Proposition ~\ref{p:1}] Let $G$ be a non abelian
subgroup of $\mathcal{H}(n, \mathcal{R})$. Since $G\backslash \mathcal{SR}_{n}\neq\emptyset$ then $\Gamma_{G}\neq\emptyset$
 and suppose that
$E_{G}$  is a vector subspace of $\mathbb{C}^{n}$, (one can replace $G$ by  $G'=T_{-a}\circ G \circ T_{a}$, for some $a\in \Gamma_{G}$).
\\
\\
(i) We will prove that there exists $a\in\Gamma_{G}$ such that
 $\overline{G(a)}=E_{G}$.  By Lemmas ~\ref{L:1},(ii) and ~\ref{L:1200}, there exist
$f=(a,\lambda),f_{1}=(a_{1},\lambda),\dots,f_{p}=(a_{p},\lambda)\in G\backslash\mathcal{SR}_{n}$  such that
$\lambda\in\Lambda_{G}\backslash\mathbb{R}$, $(a_{1},\dots,a_{p})$ and
 $(a_{1}-a,\dots,a_{p}-a)$ are two basis of $E_{G}$. Denote by $D_{k}=\mathbb{C}(a_{k}-a)+a$, $1\leq k \leq p$.
 For every $k=1,\dots, p-1$, we let $G_{k}$ be
the group generated by $f$ and $f_{k}$. One has $G_{k}\backslash\mathcal{SR}_{n}\neq\emptyset$, so by
Corollary ~\ref{CC:01},  we have  $D_{k}=\overline{G_{k}(a)}\subset
\overline{G(a)}$ for every $1\leq k \leq p$. By Lemma ~\ref{L:12}, we have
$\overline{G(a)}=E_{G}$.
\
\\
\\
 (ii) We will prove that  $\overline{G(z)}=E_{G}$ for every $z\in E_{G}$.
By (i) there exists $a\in\Gamma_{G}$ such that $\overline{G(a)}=E_{G}$.
There are two cases:\
 \\
 $\bullet$ If $G\backslash \mathcal{R}_{n}\neq\emptyset$, so by Lemma~\ref{L:3}, $\Gamma_{G}\subset \overline{G(z)}$.
 By Lemma ~\ref{L:020}.(ii), $\Gamma_{G}$ and $E_{G}$ are $G$-invariant, then
   $\overline{G(a)}\subset\overline{\Gamma_{G}}\subset \overline{G(z)}$ and so $E_{G}=\overline{G(z)}$.\
\\
$\bullet$  Suppose that  $G\subset \mathcal{R}_{n}$. Since $\overline{G(a)}=E_{G}$,  there exists a
sequence $(f_{m})_{m}\subset G$ such that $\underset{m\to +\infty}{lim}f_{m}(a)=z$. There are two situations:\
 \\
 - $f_{m}=(a_{m},\lambda_{m})\in G\backslash\mathcal{T}_{n}$ for every $m>n_{0}$, for some $n_{0}>1$.
 One has  $f_{m}(a)=\lambda_{m}a+(1-\lambda_{m})a_{m}$ and the sequence $(f_{m}(a))_{m}$
is bounded, then  the sequence $((1-\lambda_{m})a_{m})_{m}$ is bounded and so is $(a_{m})_{m}$. Therefore, there exists a subsequence
 $(\lambda_{\varphi(m)})_{m}$ of $(\lambda_{m})_{m}$ and a subsequence $(a_{\varphi(m)})_{m}$ of $(a_{m})_{m}$  such that
 $\underset{m\to+\infty}{lim}\lambda_{\varphi(m)}=\lambda'$ and
$\underset{m\to+\infty}{lim}a_{\varphi(m)}=b$, for some $b\in E_{G}$ and $\lambda'\in S^{1}$, so $\lambda'\neq 0$.
Let $f=(b,\lambda')$ if $\lambda'\neq 1$ and $f=T_{b}$ if $\lambda' =1$. Therefore $\underset{m\to+\infty}{lim}f_{m}=f$. Since $\lambda\neq0$, $f$ is invertible and
 $\underset{m\to+\infty}{lim}f^{-1}_{m}=f^{-1}$, so $f\in \overline{G}$.
As  $\underset{m\to+\infty}{lim}f_{m}(a)=z=f(a)$, we have $a=f^{-1}(z)=\underset{m\to+\infty}{lim}f^{-1}_{m}(z)$. It follows that
$a\in \overline{G(z)}$, so $E_{G}=\overline{G(a)}\subset \overline{G(z)}\subset E_{G}$,  since $E_{G}$ is $G$-invariant (Lemma ~\ref{L:020},(ii)).\
\\
- $f_{m}=T_{a_{m}}\in G\cap \mathcal{T}_{n}$ for every $m>n_{0}$, for some $n_{0}>1$. As $\underset{m\to+\infty}{lim}f_{m}(a)=z$, $\underset{m\to+\infty}{lim}a_{m}=z-a$,
 then $\underset{m\to+\infty}{lim}f_{m}=T_{z-a}\in \overline{G}$. Therefore $a=T_{a-z}(z)=\underset{m\to+\infty}{lim}f^{-1}_{m}(z)$. It follows that
$a\in \overline{G(z)}$, so $E_{G}=\overline{G(a)}\subset \overline{G(z)}\subset E_{G}$,  since $E_{G}$ is $G$-invariant (Lemma ~\ref{L:020},(ii)).\  The proof is complete.
\end{proof}
\bigskip

\begin{prop}\label{p:2} Let  $G$ be a non abelian subgroup of
$\mathcal{H}(n,\mathbb{C})$.  Suppose that $\Lambda_{G}\backslash\mathbb{R}\neq\emptyset$,
  $G\backslash\mathcal{SR}_{n}\neq\emptyset$ and $E_{G}$ is a vector space.  Then
  for every $z\in \mathbb{C}^{n}\backslash E_{G}$,  we have
   $\overline{G(z)}=\overline{\Lambda_{G}}.z+E_{G}.$
\end{prop}
\medskip

To prove the above Proposition, we need the following Lemma:
\begin{lem}\label{L:13} Let  $G$ be a non abelian subgroup of
$\mathcal{H}(n,\mathbb{C})$  such that
$G\backslash\mathcal{SR}_{n}\neq\emptyset$. For every  $b\in E_{G}$ and for  every $\lambda\in \Lambda_{G}$
 there exists a sequence $(f_{m})_{m\in\mathbb{N}}$ in $G$ such
that  $\underset{m\longrightarrow+\infty}{lim}f_{m}=f=(b,\lambda)$.
\end{lem}
\bigskip

\begin{proof} Let  $\lambda\in
\Lambda_{G}$ and  $b\in E_{G}$.   Given  $g=(a,\lambda)\in G$, so $a\in(\Gamma_{G}\cup G_{1}(0))\subset E_{G}$. By Proposition ~\ref{p:1}, we have
$\overline{G(a)}=E_{G}$.  Then there exists a sequence
$(g_{m})_{m\in\mathbb{N}}$ in $G$ such that
$\underset{m\longrightarrow+\infty}{lim}g_{m}(a)=b$. For every
$m\in \mathbb{N}$, denote by  $f_{m}=g_{m}\circ g \circ
g^{-1}_{m}$, so  $f_{m}=(g_{m}(a),\lambda)$. Hence
$\underset{m\longrightarrow+\infty}{lim}f_{m}=f$, with
$f=(b,\lambda)$.\
\end{proof}
\bigskip

\begin{proof}[Proof of Proposition ~\ref{p:2}] Let  $G$ be a non abelian subgroup of
$\mathcal{H}(n,\mathbb{C})$  such that
$G\backslash\mathcal{SR}_{n}\neq\emptyset$ and $E_{G}$ is a vector space. Let $z\in U=\mathbb{C}^{n}\backslash E_{G}$.\
\\
Let's prove that $\overline{\Lambda_{G}}.z+E_{G}\subset\overline{G(z)}$: \ Let $\alpha\in\Lambda_{G}$
and  $a\in E_{G}$.
\\
$\bullet$ Suppose that  $\alpha\in\Lambda_{G}\backslash \{1\}$. Since $E_{G}$ is a vector space, $a'=\frac{a}{1-\alpha}\in E_{G}$.
 By Lemma ~\ref{L:13} there exists a sequence $(f_{m})_{m}$ in $G$ such that
  $\underset{m\longrightarrow +\infty}{lim}f_{m}=f=(a', \alpha)\in G\backslash\mathcal{T}_{n}$. Then
  \begin{align*}
  f(z)& =\alpha (z-a')+a'\\
  \ & =\alpha z+ (1-\alpha)a'\\
  \ & =\alpha z+ a\in \overline{G(z)},
  \end{align*}
    so  $$\left(\Lambda_{G}\backslash \{1\}\right).z+E_{G}\subset \overline{G(z)}.$$\
    \\
    $\bullet$ Suppose that $\alpha=1$, by Lemma ~\ref{L:13}, there exists a sequence exists a sequence $(f_{m})_{m}$ in $G$ such that
  $\underset{m\longrightarrow +\infty}{lim}f_{m}=f=T_{a}\in G_{1}$. So  $T_{a}(z)=z+a\in \overline{G(z)}$.\
It follows that $\alpha z +a\in \overline{G(z)}$ and so $z+E_{G}\subset \overline{G(z)}.$ This proves that $\overline{\Lambda_{G}}.z+E_{G}\subset\overline{G(z)}$.\
\\
\\
 Conversely,\ let's prove that $G(z)\subset \Lambda_{G}.z+E_{G}$.\ Let $f\in G$.
 \\
 $\bullet$  Suppose that   $f=(a,\lambda)\in G\backslash \mathcal{T}_{n}$.  By Lemma ~\ref{L:020}.(i), $f(0)=(1-\lambda)a\in E_{G}$
  since $E_{G}$ is a vector space. Then $f(z)=\lambda (z-a) +a=\lambda z+ (1-\lambda)a\in \Lambda_{G}.z+E_{G}$.\
 \\
 $\bullet$  Suppose that   $f=T_{a}\in G\cap \mathcal{T}_{n}$, so $f(z)=z+a\in\Lambda_{G}.z+E_{G}$, since
 by Lemma ~\ref{L:020}.(i), $f(0)=a\in E_{G}$.
 It follows that  $G(z)\subset \Lambda_{G}.z+E_{G}$. Therefore
  $\overline{G(z)}\subset \overline{\Lambda_{G}}.z+E_{G}$. Hence $\overline{G(z)}=\overline{\Lambda_{G}}.z+E_{G}$.
\end{proof}
\bigskip

\section{{\bf Some results in the case  $G\subset \mathcal{SR}_{n}$}}
 \bigskip

In this section  $G$ is a subgroup of $\mathcal{S}_{i}\mathcal{R}_{n}$ ($i=2$ or $i=3$).

 \begin{lem}\label{L:030} Let $G$ be a subgroup of $\mathcal{H}(n, \mathbb{C})$ such that $G\subset \mathcal{S}_{i}\mathcal{R}_{n}$ $(i=2$ or $i=3)$.
 Then:\
 \\
 (i) $\Lambda_{G}=F_{i}$. Moreover, for every $\lambda,\mu\in \Lambda_{G}$, $\lambda=\mu^{k}$ for some  $k\in \mathbb{Z}$.
 \\
 (ii) $G_{1}(0)=G(0)$.\
 \\
 (iii) There exists $a\in\Gamma_{G}$, such that $G(z)=\Lambda_{G}(z-a)+G(a)$, for every $z\in\mathbb{C}^{n}$.
 \end{lem}
 \medskip

 \begin{proof} Let $a\in\Gamma_{G}$ and $G'=T_{-a}\circ G \circ T_{a}$.  Then
  $h=\mu id_{\mathbb{C}^{n}}\in G'$, for some $\mu\in \Lambda_{G}$, so $0\in\Gamma_{G}$.\
 \\
 (i) The proof follows from the construction of $\mathcal{S}_{i}\mathcal{R}_{n}$, $i=2$ or $3$ and since $F_{i}$ is cyclic.
\\
 (ii) Firstly, $\Gamma_{G'}\subset G'_{1}(0)$; Indeed, if  $f=(b,\lambda)\in G'\backslash \mathcal{T}_{n}$, then by $(i)$, $\lambda^{k}=1$ for some
  $k\in \mathbb{Z}$ since $F_{i}$ is cyclic, $i\in\{2,3\}$. Thus $f^{k}=(b,1)=T_{b}$ and so $b\in G'_{1}(0)$.\
  \\
  Secondly, let $a\in G'(0)\backslash G'_{1}(0)$ and $f=(b,\lambda)\in G'\backslash \mathcal{T}_{n}$ such that $a=f(0)=(1-\lambda)b$. By $(i)$, $\mu^{k}=\lambda$,
   for some $k\in \mathbb{Z}$. By applying Lemma ~\ref{L:bb10}.(iv) on the group $G_{k}$ generated by $h^{k}$ and $f$, we have
    $(1-\lambda)\Gamma_{G^{k}}\subset\Gamma_{G^{k}}$, so $a=(1-\lambda)b\in\Gamma_{G^{k}}\subset\Gamma_{G'}$.
    It follows that $a\in \Gamma_{G'}\subset G'_{1}(0)$.\ The proof of (ii) is complete.
    \\
   (iii) By Lemma~\ref{L:bb10+},(i), $G'(z)\subset\Lambda_{G}z+G'(0)$. Conversely, let
 $\lambda\in\Lambda_{G'}$, $a\in G'_{1}(0)$ so $T_{a}\in G'_{1}$. By (i), $\lambda=\mu^{k}$ for some $k\in\mathbb{Z}$.
 As in the proof of Lemma ~\ref{L:bb10}.(iv),  $g=T_{a}\circ h^{k}=\left(\frac{a}{1-\lambda},\ \lambda\right)$,
  so $a'=\frac{a}{1-\lambda}\in\Gamma_{G'}$. Therefore, $g(z)=\lambda(z-a')+a'=\lambda z+(1-\lambda)a'$, so
 $g(z)=\lambda z + a\in G'(z)$. Therefore,
 $\Lambda_{G'}z+G'_{1}(0)\subset G'(z)$. By (ii), $\Lambda_{G'}z+G'(0)\subset G'(z)$.  It follows that
 $G'(z)=\Lambda_{G'}z+G'(0)$, then $G(z)=\Lambda_{G}(z-a)+G(0)$, since $\Lambda_{G'}=\Lambda_{G}$, $G(z)=T_{a}(G'(z-a))$ and
 $G(a)=T_{a}(G'(0))$. The proof of (iii) is complete.
\end{proof}
\medskip

\section{{\bf Proof of main results}}\
Recall that $U=\mathbb{C}^{n}\backslash E_{G}$.
\
\\
{\it Proof of Theorem ~\ref{T:1}.} Let $a\in E_{G}$ and $G'=T_{-a}\circ G\circ T_{a}$. By Lemma ~\ref{L:1}.(iii),
 $E_{G'}=T_{-a}(E_{G})$ is a vector subspace of $\mathbb{C}^{n}$. Then :
\\
$\bullet$ The Proof of $(1).(i)$ results from  Proposition ~\ref{p:1}.
  \
  \\
  $\bullet$  Proof of $(1).(ii)$: By Proposition ~\ref{p:2},
 $\overline{G'(z-a)}=\overline{\Lambda_{G'}}.(z-a)+E_{G}'$, for every $z\in U$. So by Lemma~\ref{L:1},(ii),
 $T_{-a}(\overline{G(z)})=\overline{\Lambda_{G}}.(z-a)+E_{G}-a$,
  it follows that $\overline{G(z)}=\overline{\Lambda_{G}}.(z-a)+E_{G}$.
\
\\
$\bullet$ {\it Proof of $(2)$:} The proof of (2) results from  Lemma ~\ref{L:030}.
\hfill{$\Box$}
\
\\
\\
We will use the following Lemmas to prove Corollary ~\ref{C:1}.
\begin{lem}\label{L:15} Let $G$ be a non abelian
subgroup of $\mathcal{H}(n, \mathbb{C})$ with $G\backslash \mathcal{R}_{n}\neq\emptyset$ and $U\neq\emptyset$,
then for every $z\in \overline{G(y)}\cap U$ we have $\overline{G(z)}\cap U=\overline{G(y)}\cap U$.
\end{lem}
\bigskip

\begin{proof} Suppose that $E_{G}$ is a vector space (otherwise, by Lemma ~\ref{L:1}.(ii), we can replace $G$ by $G'=T_{-a}\circ G \circ T_{a}$
 for some $a\in E_{G}$). Let $z\in \overline{G(y)}\cap U$ and $y\in \overline{G(z)}\cap U$.
  By Theorem ~\ref{T:1}.(1).(iii), there exists $a\in E_{G}$ such that
 $\overline{G(z)}=\overline{\Lambda_{G}}(z-a)+E_{G}$. Since $E_{G}$ is a vector space and $a\in E_{G}$
 then $\overline{G(z)}=\overline{\Lambda_{G}}z+E_{G}$. In the same way, $$\overline{G(y)}=\overline{\Lambda_{G}}y+E_{G}, \ \ \ \ \ \ \ \ \ (1).$$ See that $\overline{G(z)}\cap U= (\overline{\Lambda_{G}}\backslash\{0\})z+E_{G}$.
  Write $y=\alpha z+b$, where $\alpha\in\overline{\Lambda_{G}}\backslash\{0\}$ and $b\in E_{G}$. So by (1),
$$\overline{G(y)}=\overline{\Lambda_{G}}y+E_{G}=\overline{\Lambda_{G}}(\alpha z+b)+E_{G}=\alpha\overline{\Lambda_{G}} z+E_{G}.$$
By Lemma ~\ref{L:00}, $0\in\Lambda_{G}$ and $\Lambda_{G}$  is a subgroup of $\mathbb{C}^{*}$, then $\alpha \overline{\Lambda_{G}}=\overline{\Lambda_{G}}$, since $\alpha\in\Lambda_{G}$.
 Therefore $\overline{G(y)}=\overline{\Lambda_{G}}z+E_{G}=\overline{G(z)}.$
\end{proof}
\medskip

\begin{lem}\label{L:16}  Let $G$ be a non abelian
subgroup of $\mathcal{H}(n, \mathbb{C})$ such that  $E_{G}$
 is a vector subspace of $\mathbb{C}^{n}$. Let  $z\in U$  then
 the vector subspace   $H_{z}=\mathbb{C}z\oplus E_{G}$  of
 $\mathbb{C}^{n}$  is $G$-invariant.
\end{lem}
\medskip

\begin{proof} Let  $z\in
\mathbb{C}^{n}\backslash E_{G}$  and  $f\in G$ having the form
 $f(z)=\lambda z+a$,  $z\in \mathbb{C}^{n}$, one has $a=f(0)\in
E_{G}$. For every $\alpha\in \mathbb{C}$, $b\in E_{G}$,
we have  $f(\alpha z +b)=\lambda(\alpha z +b)+a=\lambda\alpha
z+\lambda b+a$.  Since $E_{G}$ is a vector space, then
$\lambda b+a\in E_{G}$ and so $f(\alpha z +b)\in H_{z}$.
\end{proof}
\bigskip

\begin{proof}[Proof of Corollary ~\ref{C:1}]\ \\
$\bullet$ {\it The proof of $(1).(i)$:} The proof results from Lemma ~\ref{L:15}.\
\\
$\bullet$ {\it The proof of $(1).(ii)$:} As $G\backslash\mathcal{R}_{n}\neq \emptyset$, then by Lemma~\ref{L:00}, $0\in \overline{\Lambda_{G}}$. So
 the proof of (ii) results from Theorem 1,1.(1).(ii).
\\
$\bullet$ {\it The proof of $(1).(iii)$:}   Suppose that
$E_{G}$ is a vector subspace of $\mathbb{C}^{n}$ (leaving, by Lemma ~\ref{L:1}, to
replace $G$ by $G'=T_{-a}\circ G \circ T_{a}$, for some $a\in E_{G}$).

Recall that $U= \mathbb{C}^{n}\backslash E_{G}$ and let  $z, y\in U$
 with $z\neq y$.  Denote by  $H_{z}=\mathbb{C}.z\oplus E_{G}$
 and by   $H_{y}=\mathbb{C}.y\oplus E_{G}$. By lemma ~\ref{L:16} we have
 $H_{z}$  and  $H_{y}$  are $G$-invariant. Let   $\Phi:\ H_{z}\longrightarrow
 H_{y}$  be the homeomorphism defined by $\Phi(\alpha z+v)=\alpha y+v$
 for every $\Phi\in \mathbb{C}$  and  $v\in E_{G}$. For
 every  $f\in G$, with the form $f(z)=\lambda
 z+a$, $z\in\mathbb{C}^{n}$,  then by Lemma ~\ref{L:2}.(i), $a=f(0)\in E_{G}$  and so  $\Phi(f(z))=\Phi(\lambda z+a)=\lambda
 y+a=f(y)$. It follows that  $\Phi(G(z))=G(y)$.
\
\\
$\bullet$ {\it The proof of $(2)$:} The proof of (2) results from  Lemma ~\ref{L:030}.
\end{proof}
\
\\
\\
{\it Proof of Corollary ~\ref{C:2}.}\
\\
 $\bullet$  From
Corollary ~\ref{C:1}.(ii),  the closure of every orbit of $G$ contains $E_{G}$.
Since $\mathrm{dim}(E_{G})\geq 1$,  $G$ has no discrete orbit.  \hfill{$\Box$}
\
\\
\\
\\
{\it Proof of Corollary ~\ref{C:3}.} The proof of Corollary~\ref{C:3} results from Theorem ~\ref{T:1} and
Corollary ~\ref{C:1} and the fact that $\overline{U}=\mathbb{C}^{n}$ if $U\neq\emptyset$.\hfill{$\Box$}
\bigskip

\
\\
{\it Proof of Corollary ~\ref{C:4}}. If $G$ is generated by
$f_{1}=(a_{1},\lambda_{1}),\dots, f_{n-2}=(a_{n-2},\lambda_{n-2})\in \mathcal{H}(n, \mathbb{C}).$
By Lemma~\ref{L:4},
$E_{G}\subset Vect(a_{1},\dots, a_{n-2})$, so dim$(E_{G})\leq n-2$. By Theorem~\ref{T:1} there are two cases:\
\\
$\bullet$  If $G\backslash \mathcal{SR}_{n}\neq\emptyset$, then $G(z)=\Lambda_{G}z+E_{G}\subset \mathbb{C}z+E_{G}$, for every
$z\in \mathbb{C}^{n}\backslash E_{G}$ and $\overline{G(z)}=E_{G}$ for every $z\in E_{G}$. Therefore, $\overline{G(z)}\neq \mathbb{C}^{n}$.\
\\
$\bullet$ If $G\subset\mathcal{SR}_{n}$, then $\overline{G(z)}=\mathbb{C}^{n}$, for some $z\in\mathbb{C}^{n}$ if and only
 if $\overline{G_{1}(0)}=\mathbb{C}^{n}$. Since $G_{1}(0)\subset E_{G}$, it follows that $G$ has no dense orbit.

\section{{\bf Examples }}

\begin{exe} Let  $G$ be the non abelian subgroup of $\mathcal{H}(1, \mathbb{C})$ generated by
 $T_{a}$ the translation by $a\in\mathbb{C}^{*}$ and $h=e^{i\theta}I_{2}$, $\theta\notin \pi\mathbb{Z}$. Then:\
\\
(i) If $\theta\in H_{2}\cup H_{3}$ then every orbit of $G$ is  closed and discrete.\
\\
(ii) If $\theta\notin H_{2}\cup H_{3}$ then every orbit of $G$ is  dense in $\mathbb{C}$.
\end{exe}
\medskip

\begin{proof} Firstly, remark that $G$ is generated by $h$ and $g=T_{a}\circ f=\left(\frac{a}{1-e^{i\theta}},\ e^{i\theta}\right)$.\
\\
- If \ $\theta\in H_{2}\cup H_{3}$, then $g\in\mathcal{SR}_{1}$, by Theorem ~\ref{T:001}.(ii), the property (i) follows.\
\\
- If \ $\theta\notin H_{2}\cup H_{3}$, then $g\notin\mathcal{SR}_{1}$, by Theorem ~\ref{T:001}.(i), the property (ii) follows.\
\\
\end{proof}
\medskip

\begin{exe} Let  $G$  be a subgroup of  $\mathcal{H}(2, \mathbb{C})$  generated by  $f_{1}=(a_{1},\alpha_{1})$
and \ $f_{2}=(a_{2},\alpha_{2})$  and  $f_{3}=(a_{3},
\alpha_{3})$, where $\alpha_{k}\in\mathbb{C}\backslash\mathbb{R}$  with $|\alpha_{k}|\neq 1$,  for every $1\leq k \leq
3$  and  $a_{1}=\left[\begin{array}{c}
               \sqrt{2} \\
               0
             \end{array}
\right]$,  $a_{2}=\left[\begin{array}{c}
               0 \\
               1
             \end{array}
\right]$  and  $a_{3}=\left[\begin{array}{c}
               -\sqrt{3} \\
               -\sqrt{2}
             \end{array}
\right]$. Then every orbit of $G$ is dense in  $\mathbb{C}$.
\
\\
\\
Indeed, by Lemma ~\ref{L:6}.(i), $G$ is non abelian. By Proposition ~\ref{p:1}, for
every $z\in E_{G}$, we have $\overline{G(z)}=E_{G}$. By Remark~\ref{r:3} $E_{G}=\mathbb{C}^{2}$, so by Theorem ~\ref{T:1}, every orbit of
$G$ is dense in  $\mathbb{C}^{2}$.
\end{exe}
\bigskip

\begin{exe} Let   $(a_{1},\dots,a_{n})$ \ be a basis of $\mathbb{C}^{n}$ and  $\lambda\in \mathbb{C}\backslash \mathbb{R}$.
  Then every orbit of the group generated by $T_{a_{1}},\dots,T_{a_{n}}, \ \lambda Id$  is dense in $\mathbb{C}^{n}$.
\
\\
\\
  Indeed, by Remark ~\ref{r:3} we have $E_{G}=\mathbb{C}^{n}$ and by Proposition ~\ref{p:1} every orbit of $G$ is dense in $\mathbb{C}^{n}$.
 \end{exe}
  \medskip

\bibliographystyle{amsplain}
\vskip 0,4 cm

\end{document}